\documentclass[a4paper,onefignum,onetabnum]{siamart190516}
\usepackage{microtype}
\usepackage{ellipsis}
\usepackage[paper=A4,oneside,DIV=9]{typearea}
\usepackage{amsfonts}
\usepackage[english]{babel}
  \usepackage[utf8]{inputenc}
\usepackage{enumerate,expdlist,tikz,float}
\usepackage{subfig}

\usepackage{xcolor}
\colorlet{VeryRed}{red!90!black!80}

\def\soi#1#2{\left<#1,#2\right]}
\def\ci#1#2{\left[#1,#2\right]}
\def\eps{\varepsilon}

\ifpdf
  \DeclareGraphicsExtensions{.eps,.pdf,.png,.jpg}
\else
  \DeclareGraphicsExtensions{.eps}
\fi

\usepackage{pgfplots}
\pgfplotsset{width=4.2cm}
\pgfplotsset{compat=1.16}
\renewcommand{\epsilon}{\varepsilon}
\newcommand{\euler}{\mathrm{e}}

\newcommand{\R}{R}

\newcommand{\Eins}{\chi}

\newcommand{\drm}{\mathrm{d}}
\newcommand{\RR}{\mathbb{R}}

\newcommand{\NN}{\mathbb{N}}

\newcommand{\Id}{\operatorname{Id}}

\newcommand{\dom}{\operatorname{Dom}}

\newcommand{\ran}{\operatorname{Ran}}

\newcommand{\cH}{\mathcal{H}}

\newcommand{\uhat}{u^{\mathrm{opt}}}
\newcommand{\utilde}{u^{\mathrm{min}}}
\newcommand{\yhat}{y^{\mathrm{opt}}}
\newcommand{\ytilde}{y^{\mathrm{min}}}

\newcommand{\chat}{c^{\mathrm{opt}}}
\newcommand{\yhoms}{y^{*,\mathrm{hom}}}
\newcommand{\whom}{w^{\mathrm{hom}}}
\newcommand{\muhat}{\mu^\epsilon}
\newcommand{\opb}{\Psi}
\newcommand{\vb}{\psi}
\newcommand{\gammahat}{\gamma^{\mathrm{opt}}}

\newcommand{\funomega}{h}
\newcommand{\wlim}{\operatorname*{w-lim}}

\DeclareMathOperator*{\argmin}{arg\,min}
\newsiamremark{remark}{Remark}
\newsiamremark{example}{Example}

\usepackage[textsize=tiny]{todonotes}
\usepackage{cancel}
\usepackage{changes}

\usepackage{ulem}

\title{Optimal control of parabolic equations -- a spectral calculus based approach}
\author{Luka Grubi\v{s}i\'c\thanks{University of Zagreb, Department of Mathematics, Croatia 
  (\email{luka@math.hr}, \url{https://www.pmf.unizg.hr/math/luka.grubisic}).}
\and 
Martin Lazar\thanks{University of Dubrovnik, Department of Electrical Engineering and Computing, Croatia 
  (\email{mlazar@unidu.hr}, \url{http://www.martin-lazar.from.hr}).}
\and  
Ivica Naki\'c\thanks{University of Zagreb, Department of Mathematics, Croatia 
  (\email{nakic@math.hr}, \url{https://web.math.pmf.unizg.hr/\string~nakic}).}
\and
Martin Tautenhahn\thanks{Universit\"at Leipzig, Fakult\"at f\"ur Mathematik und Informatik, Germany}
(\email{martin.tautenhahn@uni-leipzig.de}, \url{https://home.uni-leipzig.de/mtau/}).}

\begin{document}

\maketitle
 \begin{abstract}
  In this paper we consider a constrained parabolic optimal control problem. The cost functional is quadratic and it combines the distance of the trajectory of the system from the desired evolution profile together with the cost of a control. The constraint is given by a term measuring the distance between the final state and the desired state towards which the solution should be steered. The control enters the system through the initial condition. We present a geometric analysis of this problem and provide a closed-form expression for the solution. This approach allows us to present the sensitivity analysis of this problem based on the resolvent estimates for the generator of the system.
The numerical implementation is performed by  exploring  efficient rational Krylov approximation techniques that allow us to approximate a complex function of an operator by a series of linear problems. Our method does not depend on the actual choice of discretization. The main approximation task is to construct an efficient rational approximation of a generalized exponential function. 
  It is well known that this class of functions allows exponentially convergent rational approximations, which, combined with the  sensitivity   analysis of the closed form solution, allows us to present a robust numerical method. Several case studies are presented to illustrate our results. \end{abstract}


\tableofcontents
%
%
%

\section{Introduction}
In this paper we consider an optimal control problem for a general linear parabolic equation governed by a self-adjoint operator on  an abstract Hilbert space. 
The task consists in identifying a control (entering the system through the initial condition) that  minimizes a given cost functional, while steering the final state at time $T>0$ close to the given target. The functional comprises of the control norm and an additional term penalizing the distance of the state from the desired trajectory.

This can be considered as an inverse problem (of initial source identification) for parabolic equations from the optimal control viewpoint. It is an important, but also  numerically challenging issue due to the dissipative nature of such equations.  It has been addressed by different methods, some including optimization and optimal control techniques \cite{FPZ95,MV07,LOT14,CVZ15}.

Optimal control problems with control in  initial conditions are less investigated than
distributed or boundary control problems.  The latter  contain controls acting along the whole time interval $[0, T]$. Such a setting is not the subject of this paper, but we refer an interested reader to  \cite{T10}, containing  a quite clear and detailed exposition of the
topic.

Our problem can be treated by exploring the  Fenchel-Rockafellar duality for convex optimisation (cf.\ \cite[Section 3.6]{Peypouquet-15}).
If the cost functional consists of the control cost only, the problem is reduced to the classical minimal norm control problem which can be treated by the Hilbert uniqueness
method.
In the seminal work \cite{CGL94} this approach is
used to transform the boundary control problems into identification problems for initial data of the adjoint heat equation, better suited to numerical methods than
the original problems. A more recent paper \cite{FM14} generalizes this method by considering cost functional including the state in addition to the control.
In order to explore efficient optimization methods, in both papers the authors approximate an original problem with a constraint on the
final state, by an unconstrained one containing a penalisation term.
The solution is then obtained by letting the penalisation constant
blow up. Similar techniques are applied in \cite{B13,FZ00}.
However,  this approach does not provide an a-priori estimate on the deviation of the final state from the given target.

In order to numerically recover the control minimising the functional of interest, most of the authors involve finite difference and/or finite element discretisation
and employ some iterative scheme (e.g. conjugate gradient), usually including the dual problem. 
Classical convex optimization techniques in Hilbert spaces (e.g. \cite{Bauschke17,Peypouquet-15}) also provide iterative
methods that can be applied to our problem. Of course, these iterative techniques come with a significant computational cost, which increases with the system dimension. 

In this paper we propose a different approach  based on the spectral calculus for self-adjoint operators and a geometrical representation of the problem. 
First, we obtain closed-form expression for the control solution as a function of the self-adjoint operator governing the dynamics of the system. This expression is almost   explicit, up to a scalar  factor ensuring that the deviation of the final state from the given target is within the prescribed tolerance.  Once the equation for this scalar unknown is solved, the method provides a direct, one-shot formula for the solution. Its numerical computation is  achieved by exploring  efficient rational Krylov approximation techniques for resolvents from \cite{BK17}, by which one constructs a rational approximant of the aforementioned function of the operator.

The proposed method and the obtained formula are given in the abstract Hilbert space framework and can be applied to optimal control problems for a large class of linear parabolic PDEs for which there exist efficient resolvent approximation algorithms. To illustrate our methods we treat optimization problems for 1D and 2D heat equations.

Our approach is an extension of the result from \cite{LazarMP-17},  where the authors  explore the spectral representation of the solution by eigenfunctions of the operator governing the system dynamics. This eventually leads to an explicit expression  (up to a scalar factor) of the optimal final state  and the optimal control.
The obtained formulae are  spectrally decoupled meaning that the $n$-th Fourier coefficient is fully expressed by the corresponding coefficients of the given data: the final target and desired trajectory. However, the practical implementation of the algorithm is constrained by the availability of the spectral decomposition of the operator. For general PDE operators with variable coefficients and/or acting on irregular domains the decomposition is in general not available or hard to construct. Also, this construction requires costly computations that can exceed the gain provided by the efficiency of the obtained formula. 

On the other side, the method proposed in this paper is applicable to more complex settings. It allows to efficiently treat PDEs with variable coefficients and defined on complicated domains. In addition, it is robust with respect to small perturbations of both the system and the cost functional, and we provide estimates on  deviation of the original solution  from the perturbed one. This is quite important in  applications, as in practice the models of interest are often not completely determined, subject to unknown or uncertain parameters, either of deterministic or  stochastic nature.  Furthermore, an expansion of a state in eigenfunctions typically converges rather slowly except in very specific cases. In comparison, representations of solutions using Krylov subspaces are much more efficient in the number of required terms.
  
The paper is organised as follows. In the next section we formulate the problem and state the main result (Theorem \ref{thm:solution_final_state}). In Section~\ref{sec:sensitivity_analysis} we present the sensitivity analysis which  justifies a finite-dimensional approximation of the problem.  In Section~\ref{Sec:rat} we present the rational function approximation theory and discuss the stability of the finite element approximation of the problem. Further, we discuss the relationship between numerical rational functions calculus as realized by the \texttt{rkfit} algorithm \cite{BK17} and the approximation problem for the generalized exponential functions which appear as central for the study of the concrete numerical examples. In Section~\ref{sec:examples} we present 1D and 2D numerical examples which are outside the scope of the original eigendecomposition method from \cite{LazarMP-17}. Within the concluding remarks we discuss efficiency of the introduced method, open perspectives and comparison to other approaches.

%
%
%
%
%
%
%
\section{Setting of the problem and characterisation of the solution}
Let $\cH$ be a Hilbert space and $A$ be an upper-bounded self--adjoint operator in $\cH$ with an upper bound $\kappa$, i.e.\ $\max \sigma (A)\le \kappa$. We denote by $(S_t)_{t \geq 0}$ the semigroup generated by $A$. We consider for $f\in L^2((0,\infty);\cH)$ and $u \in \cH$ the Cauchy problem
\begin{equation} \label{eq:Cauchy}
 \begin{cases}
 y' (t)  = A y (t) + f(t),  & t > 0 ,  \\
 y (0) = u .
 \end{cases}
\end{equation}
Note that the mild solution of \eqref{eq:Cauchy} is given by 
\[ 
y(t) = S_t u + \int_0^t S_{\tau}f(t-\tau)\drm \tau ,\quad t \geq 0,
\]
and is an element of $L^2((0,\infty);\cH$), see e.g.\ \cite{EngelN1999}. 
We say that the system \eqref{eq:Cauchy} is controllable to a target state $y^* \in \cH$ in time $T > 0$ if there is $u \in \cH$ such that 
\[
 S_T u +  \int_0^T S_{\tau}f(T-\tau)\drm \tau = y^* .
\]
We say that the system \eqref{eq:Cauchy} is approximately controllable in time $T > 0$ if for all $y^* \in \cH$ and all $\epsilon > 0$ there exists $u \in \cH$ such that
\begin{equation} \label{eq:approx}
 \Bigl\lVert S_T u + \int_0^T S_{\tau}f(T-\tau)\drm \tau  - y^* \Bigr\rVert \leq \epsilon .
\end{equation}

\begin{remark} \label{St-properties}
Let us note that for any $T>0$ the  operator $S_T$ is injective with a dense range. For the reader's convenience, we present a short proof. For the injectivity assume that there exists $0 \ne x \in \cH$ such that $S_T x = 0$. The semigroup property immediately implies that $S_t x = 0$ for all $t > T$. Let $0 < t < T$ be arbitrary. Then there exists $k\in \mathbb{N}$ such that $2^k t > T$ and hence $S_{2^k t} x = 0 $. Since $S_t$ is a non--negative operator, we have $0 = \langle S_{2^k t} x, x\rangle = \lVert S_{2^{k-1}t} x\rVert^2 $, hence $S_{2^{k-1}t} x = 0$. Now by induction it follows $S_t x = 0$. Since $( S_t )_{t\geq 0}$ is a $C_0$-semigroup, we obtain $0 = \lim_{t\to 0} S_t x = x$, a contradiction. Note that this also implies that $\ran(S_T)$ is a dense subspace of $\cH$ as $S_T$ is a self--adjoint operator. 
\end{remark}
\medskip

Since the range of $S_T$ is dense in $\cH$  the system \eqref{eq:Cauchy} is indeed approximately controllable in any time $T > 0$.
In the class of initial values satisfying \eqref{eq:approx} for a given target $y^* \in \cH$ and time $T > 0$,  we are looking for those with minimal cost. More precisely, for $\epsilon,T > 0$ and $y^* \in \cH$ we introduce the problem
\begin{equation} \label{eq:problem}
 \min_{u \in \cH} \left\{ J(u) \colon \Bigl\lVert S_T u + \int_0^T S_{\tau}f(T-\tau)\drm \tau - y^* \Bigr\rVert \leq \epsilon \right\}
\end{equation}
where
\[
 J (u) =  \frac{\alpha}{2} \lVert u \rVert^2 + \frac{1}{2} \int_0^T \beta (t) \Bigl\lVert S_t u + \int_0^t S_{\tau}f(t-\tau)\drm \tau - w(t) \Bigr\rVert^2 \drm t ,
\]
$\alpha > 0$ and  $\beta \in L^\infty ((0,T) ; [0,\infty))$ are weights of the cost, and $w \in L^2 ((0,T) ; \cH)$ is the target trajectory. 

Of course, the notation used  can be somewhat simplified by substituting $S_t u + \int_0^t S_{\tau}f(t-\tau)\drm \tau $ with $y(t)$, but we want to keep the formulation that explicitly shows  dependence of the problem on  the given data  $f, y^*$ and $w$, and the unknown control $u$.

For $\epsilon > 0$ and $x \in \cH$ we denote by $B_\epsilon (x) = \{ y \in \cH \colon \lVert y-x \rVert \leq \epsilon \}$ the closed ball of radius $\epsilon$ and center $x$. Our problem \eqref{eq:problem} can be restated as
\begin{equation} \label{p2}
\min_{u\in \cH} \left\{J\left(u\right) + I_{B_\epsilon (y^*)}\left(\mathcal{S}_T u+ \int_0^T S_{\tau}f(T-\tau)\drm \tau \right) \right\},
\end{equation}
where $I_{B_\epsilon (y^*)}$ is the corresponding \emph{indicator function} defined as
\begin{equation*} 
I_{B_\epsilon (y^*)}  \left(y\right)=
\begin{cases}
0 &\mbox{if } y\in  B_\epsilon (y^*), \\
+\infty &\mbox{else}.
\end{cases}
\end{equation*}
Since the function $u \to J(u)+I_{B_\epsilon (y^*)}\circ(\mathcal{S}_T u+ \int_0^T S_{\tau}f(\tau)\drm \tau) $ is proper, strongly convex and lower-semicontinuous, problem \eqref{eq:problem} has a unique solution, which we denote by $\uhat $ (see, for instance, \cite[Corollary 2.20]{Peypouquet-15}). 
Moreover, we define
\[
 \utilde  = \argmin_{u \in \cH} J (u) ,
\]
as the solution to the corresponding unconstrained problem, while by 
$\ytilde  = S_T u^{\mathrm{min}} + \int_0^T S_{\tau}f(T-\tau)\drm \tau$ and $\yhat  = S_T \uhat  + \int_0^T S_{\tau}f(T-\tau)\drm \tau$ we denote the corresponding optimal final states obtained from $\utilde $ and $\uhat $, respectively. 

\begin{remark} 
 Regarding the results we use from \cite{Peypouquet-15}, we note that they are stated in the case of real Hilbert spaces only. However, they carry over to the complex case by realifying the Hilbert space $\cH$ and taking the real part of the inner product instead of the (complex) inner product.
\end{remark} 
\medskip

The problem \eqref{p2} has a form of the composite optimization problem \cite{Peypouquet-15}. That is to say that the target functional is a sum of a quadratic function and a ``simple'' function, e.g. a composition with an indicator function. Such problems -- when posed in the correct abstract setting -- are typically solved by methods based on proximal operator. Instead, we use the spectral calculus to explicitly construct an operator theoretic representation of the trajectories, cf.\ Remark~\ref{rem:2.8} and Section~\ref{Sec:rat}. Interestingly, the structure of the abstract composite optimization problem is still preserved in the solution formula. We will comment on this explicitly in Remark~\ref{rem:proximal} after the statement of the main theorem.

\par
We define $\yhoms = y^* - \int_0^T S_{\tau}f(T-\tau)\drm \tau$ and $\whom = w - \int_0^{\boldsymbol{\cdot}} S_{\tau}f(\boldsymbol{\cdot} - \tau)\drm \tau$. Then our problem~\eqref{eq:problem} can be written as
\begin{equation}
\label{eq:problemh}
   \min_{u \in \cH} \left\{ J(u)  \colon \lVert S_T u  - \yhoms \rVert \leq \epsilon \right\}, 
\end{equation}
where
\[
   J (u) =  \frac{\alpha}{2} \lVert u \rVert^2 + \frac{1}{2} \int_0^T \beta (t) \lVert S_t u - \whom(t) \rVert^2 \drm t .
\]
Note that $S_t u$ is the solution of the corresponding homogeneous Cauchy problem with $f=0$. 
\par
If $\epsilon \geq \lVert \ytilde - y^* \rVert$ it follows that the solution of \eqref{eq:problemh} (and hence also of \eqref{eq:problem}) satisfies $\uhat = \utilde$. The following theorem covers the non-trivial case $0 < \epsilon < \lVert \ytilde - y^* \rVert$ as well.

\begin{theorem} \label{thm:solution_final_state}
 Let $T,\epsilon > 0$ and $y^* \in \cH$. 
 Then the optimal initial state $\uhat$ is given by
  \begin{equation} 
 \label{sol_form}
 \uhat = ( \muhat S_{2T} + \opb)^{-1} (\muhat S_T \yhoms +  \vb) ,
\end{equation}
where
\[
 \opb = \alpha \Id  + \int_0^T \beta (t) S{}_{2t} \drm t, 
 \quad
 \vb = \int_0^T \beta (t) S_{t} \whom(t) \drm t ,
\]
and $\muhat \ge 0$ is the unique solution of $\Phi (\mu) =  \epsilon$ if $\epsilon < \lVert \ytilde - y^* \rVert = \lVert \Psi^{-1}S_T \psi - \yhoms \rVert $, and zero otherwise. 
Here 
$\Phi \colon [0,\infty) \to [0,\infty)$ is the function defined  by 
  \begin{equation} \label{eq:Phi}
   \Phi(\mu) = \lVert \yhoms -  (\mu S_{2T} +  \opb)^{-1} (\mu S_{2T} \yhoms + S_T\vb) \rVert .
  \end{equation}
\end{theorem}
\begin{remark}
\label{u_min}
 Since $\Psi$ is positive definite, we indeed have that $\Psi$ and $\mu S_{2T} + \Psi$, $\mu \geq 0$, are invertible. 
 Moreover, for the functional $J$ we have $\nabla J(u)= \Psi u - \psi$. As $\utilde$ is its global minimizer, it immediately  follows  that $\utilde = \opb^{-1}\vb$, and thus $\lVert \Psi^{-1}S_T \psi - \yhoms \rVert = \lVert \ytilde - y^* \rVert$. 
\end{remark} \medskip
\begin{remark}
For $\epsilon <    \lVert \ytilde - y^* \lVert$ we obtain from \eqref{sol_form} and \eqref{eq:Phi}
\[
 \Phi(\muhat) = \lVert y^* -  \yhat \rVert = \epsilon.
 \]
 In other words, if the global minimizer $\utilde$ of the unconstrained problem does not drive the system to the target ball $B_\epsilon (y^*)$, then  the optimal final state lies on the boundary of this ball, cf.\ Lemma~\ref{lem:yopt}. This is in accordance with previous results on similar problems (e.g. \cite[Proposition 2.1]{LazarMP-17} and \cite[Theorem 2.4]{CVZ15}) that provide  the same characterisation of the optimal solution.  
\end{remark} \medskip

\begin{remark} 
Let $\phi(\mu) = \yhoms -  (\mu S_{2T}+  \opb)^{-1} (\mu S_{2T} \yhoms + S_T\vb)$, hence $\Phi(\mu) = \lVert \phi(\mu) \rVert $. Then $\phi(\mu) = \yhoms - x$,  where $x$ is the solution of the equation
\begin{equation}
\label{eq:linsym-mu}
\left(\mu S_{2T} + \opb \right) x =  \mu S_{2T }\yhoms + S_T \vb,
\end{equation}
hence the calculation of $\Phi(\mu)$ reduces to solving a linear equation.
Note also that the optimal initial state $\uhat$ is the solution of the equation
\begin{equation}
\label{eq:linsym-final}
\left(\muhat S_{2T} +\opb \right) x = \muhat S_{T} \yhoms + \vb.
\end{equation}
\end{remark} \medskip
\begin{remark}
\label{rem:proximal}
 In order to give some geometrical intuition for the constrained optimization problem that we solve in the Hilbert space setting, let us  observe a formal similarity of the result of Theorem \ref{thm:solution_final_state} and the known finite dimensional result \cite{ProximalOP}. Let $\lambda>0$ and let $A\in\mathbb{R}^{m\times n}$ be of full rank. Then for a given $x \in \RR^n$
  \[
  u_x = \argmin_{u \in \RR^n}\left\{\lambda\lVert Au \rVert +\frac{1}{2} \lVert u-x \rVert^2\right\}
  \]
  is given by the formula
  \[
  u_x
  =
  \begin{cases}
  x-A^*(AA^*)^{-1}Ax,& \text{if} \ \lVert (AA^*)^{-1}Ax \rVert \leq\lambda, \\
  x-A^*(AA^*+\alpha^* \Id)^{-1}Ax& \text{if} \ \lVert (AA^*)^{-1}Ax \rVert >\lambda.
  \end{cases}
  \]
  Here $\alpha^*$ is the unique positive root of the decreasing function
  \[
  \phi(\alpha)= \lVert (AA^*+\alpha \Id)^{-1}Ax \rVert^2-\lambda^2 .
  \]  
  The function $\phi$ takes the roll of $\Phi$ from \eqref{eq:Phi}.
\end{remark} \medskip
\begin{remark}\label{rem:2.8}
To calculate $\opb$ we will use the fact that it can be written as a function of $A$ by using
\[
\int_0^T \beta(t) S_{2t} \drm t = \int_{- \infty}^{\infty} \int_0^T \beta(t) \exp(-2 t \lambda) \drm t \, \drm (E(\lambda)) = \tilde \beta_0(A),
\]
where $\tilde \beta_0$ is a function given by $\tilde \beta_0(\lambda) = \int_{0}^{T} \beta(t) \exp(- 2 t \lambda) \drm t$ and $E$ is the spectral measure of $A$.

However such approach does not work directly with other term entering the formula for the  solution \eqref{sol_form}. Namely, it is not possible to find a nice closed formula for $\vb$ except in special situations. But we can always find a good approximant for $\vb$. Let $\tilde w (t) = \sum_{i=1}^N w_i \chi_{[t_{i-1},t_i]} $ be an approximation of $w$, where $0 = t_0 < t_1 < \cdots < t_N = T$, $w_i \in \cH$, $i=1,\ldots, N$, and $\chi_S$ is the characteristic function of the set $S$. Then 
\[
\tilde\vb = \sum_{i=1}^N \tilde \beta_i(A) w_i, \text{ where } \tilde \beta_i (\lambda) =  \int_{t_{i-1}}^{t_i} \beta(t) \exp(- t \lambda) \drm t, 
\]
is an approximation of $\vb$.

If the function $\beta$ is such that we can not explicitly calculate $\tilde \beta_i$, $i = 0,\ldots,N$, we can still find appropriate approximations of $\opb$ and $\vb$ by finding appropriate approximations of $\tilde \beta_i$, $i = 0,\ldots,N$.
 \end{remark} \medskip
 
 Before proving Theorem \ref{thm:solution_final_state} we provide two auxiliary results. 
 \begin{lemma}
 \label{lem:yopt}
 Let $\epsilon > 0$ and $\lVert \ytilde-y^* \rVert > \eps$, then the optimal final state verifies $\lVert y^*  - \yhat \rVert = \epsilon$, i.e. $\yhat$ lies on the boundary of the target ball. 
  \end{lemma}
  \begin{proof}
  Let us suppose the contrary, that $\yhat \in B_{\eps}(y^*)$. Then it exists $\eta>0$ such that $B_{\eta}(\yhat )\subset B_{\eps}(y^*)$ and then, by continuity of $S_T $, a $\delta>0$ such that $S_T B_{\delta} (\uhat)+ \int_0^T S_{\tau}f(T-\tau)\drm \tau \subset B_{\eta}(\yhat ) \subset B_{\eps}(y^*)$. In particular, every $u\in B_{\delta} (\uhat)$ is a feasible control for the problem \eqref{eq:problem}. As $\uhat$ is the solution of the same problem, it holds  $J(\uhat)\leq J(u)$ for every $u\in B_{\delta} (\uhat)$. But, by the convexity of $J$, a local minimizer is also global. Then $\uhat$ is solution of the unconstrained problem, which contradicts the assumption $\epsilon < \lVert \ytilde - y^* \rVert$. 
 \end{proof}

  \begin{lemma}
  \label{lem:Phi}
The function $\Phi$ has the following properties:
  \begin{enumerate}[(a)]
    \item If $0  < \lVert y^* -\ytilde \rVert$, then $\Phi$ is a strictly decreasing function,
    \item $\lim_{\mu\to \infty}\Phi (\mu) =  0$,
    \item $ \Phi (0) =  \lVert  y^* -\ytilde  \rVert$.
  \end{enumerate}
\end{lemma}
\begin{proof}
   Note that the function $\Phi$ can be rewritten as 
   \[\Phi(\mu) = \lVert ( \mu S_{2T} + \opb)^{-1} (\opb  \yhoms - S_T\vb) \rVert.\]
For the derivative of $\mu \mapsto \Phi^2(\mu)$ we have
\begin{align*}
 (\Phi^2)' (\mu) &= -2 \bigl\langle S_{2T} (\mu S_{2T} + \opb)^{-2} ( \opb \yhoms - S_T\vb) , (\mu S_{2T} +   \opb)^{-1} ( \opb \yhoms - S_T\vb) \bigr\rangle \\
 &=  -2 \bigl\lVert S_T (\mu S_{2T} + \opb)^{-3/2} ( \opb \yhoms - S_T\vb)  \bigr\rVert^2 \leq 0  .
\end{align*}
As $(\Phi^2)' = 2 \Phi \Phi'$ and $\Phi$ is a nonegative function, it follows 
that  $\Phi' (\mu) \le  0$ for all $\mu > 0$ and we have strict negativity if $ \opb \yhoms \ne S_T\vb$.  Recall that 
$ \opb \utilde =\vb$ (cf. Remark \ref{u_min}), hence 
$ \opb \ytilde =S_T \vb$. Now $ \opb \yhoms = S_T\vb$ would imply $ \opb\ytilde =  \opb y^*$, a contradiction with the assumption and the invertibility of $ \opb$. 

We now prove (b).
First note
\[
y - (\mu S_{2T} + \opb)^{-1} \mu S_{2T} y = y - \mu (\mu + S_{2T}^{-1} \opb )^{-1}y =  (\mu + S_{2T}^{-1} \opb)^{-1} S_{2T}^{-1} \opb  y  
\]
for all $y\in \dom(S_{2T}^{-1} \opb )$, which implies that we have
$
y - \mu (\mu + S_{2T}^{-1} \opb )^{-1}y = 0
$
for all $y\in \dom(S_{2T}^{-1} \opb )$. Let $(y_n)$ be a sequence from $\dom(S_{2T}^{-1} \opb )$ which converges to $\yhoms$. Let $n\in \mathbb{N}$ be arbitrary. Using $\lVert (\mu + S_{2T}^{-1} \opb)^{-1} \rVert = \mathop{\mathrm{dist}}(\mu, -\sigma (S_{2T}^{-1} \opb))^{-1}\le \mu^{-1} $ it follows
\begin{align*}
&\lim_{\mu\to \infty}  \lVert \yhoms - (\mu S_{2T} + \opb)^{-1} \mu S_{2T} \yhoms \rVert \\ 
&= \lim_{\mu\to \infty} \lVert \yhoms - y_n - \mu (\mu + S_{2T}^{-1} \opb)^{-1}(\yhoms - y_n) + y_n -  \mu (\mu + S_{2T}^{-1} \opb)^{-1}y_n \rVert  \\
&\le 2\lVert \yhoms - y_n \rVert  + \lim_{\mu\to \infty} \lVert y_n -  \mu (\mu + S_{2T}^{-1} \opb)^{-1}y_n   \rVert = 2\lVert \yhoms - y_n \rVert ,
\end{align*}
and taking the limit $n\to \infty$ we obtain
\[
\lim_{\mu\to \infty}  \lVert \yhoms - (\mu S_{2T} + \opb)^{-1} \mu S_{2T} \yhoms \rVert = 0.
\]

Hence to prove (b) we only have to show
\begin{equation}
\label{eq:b-limit}
\lim_{\mu\to \infty} \lVert (\mu S_{2T} +   \opb)^{-1} S_T\vb \rVert = \lim_{\mu\to \infty} \lVert S_T^{-1} (\mu  +  S_{2T}^{-1} \opb)^{-1} \vb \rVert =  0.
\end{equation}
Since $\ran (S_T)$ is dense in $\cH$ (cf. Remark \ref{St-properties}), there exists  a sequence  $(\vb_m)_{m \in \NN}$  in $\ran (S_T)$ such that $\lim_{m\to \infty}\vb_m =\vb$, and let $v_m\in \cH$ be such that $\vb_m = S_T v_m$. Then $\lim_{\mu \to \infty}S_T^{-1} (\mu  +  S_{2T}^{-1} \opb)^{-1}\vb_m = \lim_{\mu \to \infty}(\mu  +  S_{2T}^{-1} \opb)^{-1} v_m =0$ for all $m\in \mathbb{N}$
and for $\vb$ we have the following estimate
\begin{equation}
\label{eq:bn}
\lVert S_T^{-1} (\mu  +  S_{2T}^{-1} \opb)^{-1} \vb \rVert  \leq \lVert  S_T (\mu S_{2T} +   \opb)^{-1} \rVert  \lVert\vb -\vb_m \rVert  + \lVert S_T^{-1} (\mu  +  S_{2T}^{-1} \opb)^{-1}\vb_m \rVert  .
\end{equation}
By differentiating one can show that the mapping $[0,\infty) \ni \mu \mapsto \lVert (\mu S_{2T} +   \opb)^{-1}x \rVert^2 $ is a decreasing function for all $x\in\cH$, so in particular 
$[0,\infty) \ni \mu \mapsto \lVert S_T (\mu S_{2T} +   \opb)^{-1} \rVert$ is a bounded function. 
Thus we can pass to the limit in \eqref{eq:bn}, from which we obtain \eqref{eq:b-limit}.

Finally, (c) follows from $ \Phi (0) =  \lVert \yhoms - \opb^{-1}S_T\vb  \rVert =  \lVert y^* - \ytilde  \rVert$.
\end{proof}

Now we are ready to provide  proof of Theorem  \ref{thm:solution_final_state}.

\begin{proof}[Proof of Theorem~\ref{thm:solution_final_state}]
Note that the  case $\epsilon \ge \lVert \ytilde - y^* \rVert$ is covered trivially. Indeed,  by choosing $\muhat = 0$ in  \eqref{sol_form}
we obtain $\uhat = \Psi^{-1}\psi = \utilde$. By assumption on $\epsilon$, the unconstrained minimizer $\utilde$ is admissible, and clearly optimal.

For the rest of the proof we consider the case  $0 < \epsilon < \lVert \ytilde - y^* \rVert$. We fix $T,\epsilon > 0$ and $y^* \in \cH$.

Based on Lemma \ref{lem:yopt}, our problem \eqref{eq:problem} (or its equivalent form \eqref{eq:problemh}) can be restated as 
\begin{equation*} 
 \min_{u \in \cH} \left\{ J(u) \colon \lVert S_T u  - \yhoms \rVert =\epsilon \right\}.
\end{equation*}
whose  associate Lagrange functional reads as
\[
L(u, \mu)= J(u) + \frac{\mu}{2} \left(\lVert S_T u - \yhoms \rVert^2 - \epsilon^2\right).
\]
Its (global) minimizer corresponds to the unique solution of our problem. 
As $J$ is a differentiable function, so it is the Lagrangian $L$. Thus it achieves the minimum value in the point $(\uhat, \muhat)$ satisfying
\[
\nabla_{u, \mu} L (\uhat, \muhat)= 0.
\]
By exploring the relation $\nabla J(u)= \Psi u - \psi$ we get
\[
\nabla_{u} L (\uhat, \muhat) = \Psi \uhat - \psi + \muhat (S_{2T } \uhat  - S_T \yhoms) =0,
\]
which directly leads to the formula \eqref{sol_form}. In order to determine the optimal value of the Lagrange multiplier, we use Lemma  \ref{lem:yopt} providing
$
\lVert S_T \uhat  - \yhoms \rVert =\eps
$.
Plugging the expression for the optimal control $\uhat$ we obtain
\[
\Phi(\muhat)=\lVert \yhoms -  (\mu S_{2T} +  \opb)^{-1} (\muhat S_{2T} \yhoms + S_T\vb) \rVert =\eps.
\]
Lemma \ref{lem:Phi} and the assumption  $\epsilon < \lVert \ytilde - y^* \rVert$ ensure the existence of the unique value $\muhat$ satisfying the last relation, which completes the proof. 
\end{proof}
For the rest of this section  we provide  a geometric interpretation of the final optimal state. 
For $x \in \cH$ and $W\subset \cH$ closed and convex, we denote by $\Pi_W (x)$ the unique projection of $x$ onto the set $W$. 

For $c \in \mathbb{R}$ we denote by $\Gamma_c (g)$ the $c$-sublevel set of a function $g \colon \dom (g)\subset \cH \to \RR$  defined by 
\[ 
\Gamma_c (g) = \{y \in \dom(g) \colon g (y) \leq c\}. 
\]

Note that the sublevel set of a convex function is convex, see, e.g., Remark~2.10 in \cite{Peypouquet-15}.  In particular, we have the following result. 
\begin{lemma}\label{lemma:projection}
 Let $g : \cH \to \RR$ be a differentiable convex function, $c > \inf g$, and $y_1 \in \cH \setminus \Gamma_c (g)$. Denote by $y_0 \in \Gamma_c (g)$ the unique projection of $y_1$ onto $\Gamma_c (g)$. Then there is $\gamma > 0$ such that
 \[
  y_1 - y_0 = \gamma \nabla g(y_0) .
 \]
\end{lemma}
\begin{proof}
First note that the projection  $y_0$ is the solution of the constrained minimisation problem
\[
\min_{y\in \Gamma_c (g)}  \Big\{ \frac 12 \lVert y-y_1 \rVert^2: \; y \in \partial\Gamma_c (g)\Big\},
\]
whose Lagrangian is given by the relation 
\[
L(y,\lambda)=\frac 12 \lVert y-y_1 \rVert^2 + \lambda (g(y)-c).
\]
 Since $g$ is a differentiable, real-valued function,  the associated Lagrangian is differentiable as well.
Therefore its unique minimizer $(y_0, \gamma)$ satisfies $\nabla L(y_0, \gamma) =0$, which is exactly the formula appearing in the statement of the lemma.

The sign of $\gamma$ results from the fact that $g(y_1) > g(y_0)$ (as $y_1 \in \cH \setminus \Gamma_c (g))$.
\end{proof}
We introduce the functional $h: \ran(S_T) \to \mathbb{R}$, defined by $h(y)=J(S_T^{-1} y)$.
For each $c \in \RR$ we define the corersponding sublevel set
\begin{align*}
 W_{c} &:=\{S_T u \colon u \in\cH \ \text{with}\ J (u) \leq c\} = \Gamma_c (h) .
\end{align*}
With the notation $c_{\mathrm{min}} = J (u^{\mathrm{min}})$ we have that $W_{c} = \emptyset$ if and only if $c < c_{\mathrm{min}}$. 
In order to show that $W_c$ are closed we will use the fact that the weak and strong closure of a convex set agrees. If $c < c_{\mathrm{min}}$ the set $W_c$ is empty and thus closed. Assume now $c \geq c_{\mathrm{min}}$ and let $(y_n)_{n \in \NN}$ be a (strongly) convergent sequence in $W_c \subset \cH$. We denote by $y = \lim_{n\to \infty} y_n$ its limit. If we show that $y\in W_c$, i.e.\ $h (y) \leq c$, we are done. For $n \in \NN$ let $x_n \in \cH$ be such that $y_n = S_T x_n$. As $J(x_n)\le c$ for all $n \in \NN$, it follows that $(x_n)_{n \in \NN}$ is a bounded sequence. Hence there exists a weakly convergent subsequence, still denoted by $(x_n)_{n \in \NN}$. We denote by $x = \wlim_{n\to \infty}x_n$ its weak limit.
For all $z\in \cH$ we have $ \langle S_T x_n -y, z \rangle = \langle x_n, S_T z \rangle - \langle y, z \rangle \to \langle x, S_T z \rangle - \langle y, z \rangle  $. Hence 
 $S_T x = y$. 
From the continuity of the function $J$ and $h (y_n) \leq c$ for each $n \in \NN$ we conclude
\[
 c \geq \wlim_{n \to \infty} h (y_n) 
 = 
  J ( \wlim_{n \to \infty} x_n) 
 = 
  J (x)
 = 
  h (y) .
\]
If we formally differentiate the function $h$, for the gradient of $h$ we obtain
 \begin{equation}
 \begin{aligned}
\label{eq:h-gradient}
\nabla h(y)&=  S_{2T}^{-1} \left(  \alpha y + \int_0^T \beta (t) S_{2t}y \,\mathrm{d}t \right)  - S_{T}^{-1} \int_0^T \beta(t) S_{t} \whom(t) \,\mathrm{d}t \\
&= S_{2T}^{-1} \Psi y - S_{T}^{-1} \psi.
\end{aligned}
\end{equation}
Below we comment how to tackle the case if $h$ is not differentiable in $\cH$.
Denoting by $\chat=J(\uhat) $, from Lemma \ref{lem:yopt} it follows that the optimal final state belongs to $W_{\chat}\cup \partial B(y^*, \eps)$. Specially, it equals the projection $ \Pi_{W_\chat} (y^*)$ of the target state to the sublevel set $W_{\chat}$ (Figure \ref{fig:illustrationWc}). 

By putting $c = \chat$ and $y_1 = y^*$ into   Lemma \ref{lemma:projection}, we obtain that there exists $\gammahat > 0$  such that 
 \begin{equation*} 
  y^* - \yhat = \gammahat \nabla \funomega (\yhat) .
 \end{equation*}
By acting with $S_T$ on the last equality,  taking into account that $\yhat = S_T \uhat+  \int_0^t S_{\tau}f(t-\tau)\drm \tau$ and using \eqref{eq:h-gradient} we obtain 
\begin{equation*} 
S_T \yhoms -  S_{2T} \uhat = \gammahat (\Psi \uhat - \psi) ,
\end{equation*}
which lead us again to the solution formula \eqref{sol_form}  with $\muhat= 1/\gammahat$. 

One can bypass the problem of non-differentiability of $h$ by taking a sequence of approximate functionals $h_n(y) = J(\mathrm{e}^{-TA_n}y)$ for $n$ large enough, where $A_n$ are the Yosida approximations of the operator $A$. Then $h_n:\cH \to \mathbb{R}$ are differentiable functions which converge to $h$. Instead of \eqref{eq:problemh} for all  $n$ large enough we study the problem
\begin{equation}
  \label{eq:n-constr_problem}
  \min_{y\in \cH} \left\{ \funomega_n(y)\colon \lVert y-y^* \rVert  \leq \epsilon \right\}. 
\end{equation}
Then one can prove the corresponding version of Lemma \ref{lem:yopt} and use Lemma \ref{lemma:projection} to show that for all $n$ large enough we have
\begin{equation*} 
\mathrm{e}^{TA_n} y^* -  \mathrm{e}^{2TA_n} \uhat_n = \gammahat_n (\Psi \uhat_n - \psi) ,
\end{equation*}
where $\uhat_n$ is the unique solution of \eqref{eq:n-constr_problem}. Finally, by proving $\lim_{n\to \infty}\uhat_n = \uhat$ and $\lim_{n\to \infty}\gammahat_n = 1/\muhat$, one recovers \eqref{sol_form}. We skip the details.
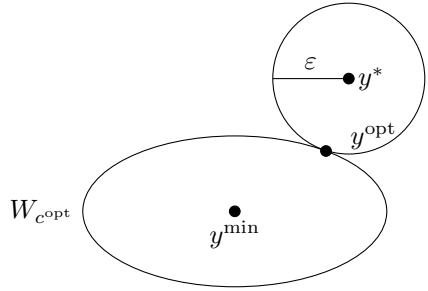
\begin{figure}[ht] \centering
\begin{tikzpicture}
 \draw (0,0) ellipse (2cm and 1cm);
 \draw (-2,0) node[left] {$W_{\chat }$};
 \draw (1.5,1.76) circle (1cm);
 \filldraw (1.5,1.76) circle (2pt) node[right] {$y^*$};
 \filldraw (1.2,0.8) circle (2pt);
 \draw (1.4,1) node[right] {$\yhat$};
 \draw (1.5,1.76)--(0.5,1.76);
 \draw (1,1.76) node[above] {$\epsilon$};
 
 \filldraw (0,0) circle (2pt) node[below] {$\ytilde$};
 \end{tikzpicture}
\caption{Illustration of the optimal final state. It equals the projection of the target state to the sublevel set $W_{\chat}$.}
\label{fig:illustrationWc}
\end{figure}
\section{Sensitivity analysis} 
\label{sec:sensitivity_analysis}
%
%
In this section we show that the solution of \eqref{eq:problem} is stable in the sense that if the parameters $\alpha$, $f$, $\beta$, $w$, $y^*$ and $A$ are perturbed by a small perturbation, then the solution of the perturbed problem is as well a small perturbation of the solution of the unperturbed problem. 
Let $0 < \nu <1 $ and 
\begin{enumerate}[(i)]
 \item $\delta\alpha < \nu$ such that $\alpha + \delta\alpha > 0$,
 \item $\delta f \in L^2 ((0,\infty) ; \cH)$ such that $\lVert \delta f \rVert_{L^2 ((0,\infty) ; \cH)} < \nu$ ,
 \item $\delta \beta \in L^\infty ((0,T); \RR)$ such that $\lVert \delta \beta \rVert_{L^{\infty}((0,T);\mathbb{R})} < \nu $ and $\beta + \delta \beta \in L^\infty ((0,T) ; [0,\infty))$,
  \item $\delta w \in L^2 ((0,T);\cH)$ such that $\lVert \delta w\rVert_{L^2((0,T);\cH )} < \nu$,
  \item $\delta y^* \in \cH$ such that $\lVert \delta y^* \rVert < \nu$.
  \item For the perturbation $\delta A$ of the operator $A$ we assume that $\delta A$ is a symmetric linear operator in $\cH$, $A + \delta A$ is an upper bounded self-adjoint operator in $\cH$, and there exists $\zeta > \max \{\max \sigma (A) , \max \sigma (A + \delta A)\}$ and $R > 0$ such that for all $s \in \RR$ we have 
\begin{equation}
\label{eq:deltaAcondition}
\lVert (\zeta + \mathrm{i} s- A - \delta A)^{-1} - (\zeta + \mathrm{i} s  -A)^{-1} \rVert <  \nu  \begin{cases} 1, & |s|\le R, \\ s^{-2}, & |s| > R. \end{cases}
\end{equation}
\end{enumerate}
We denote by $(S_t^\delta)_{t \geq 0}$ the semigroup generated by $A + \delta A$, and introduce the short hand notation $\alpha_\delta = \alpha + \delta \alpha$, $f_\delta = f + \delta f$, $\beta_\delta = \beta + \delta \beta$, $w_\delta = w + \delta w$ and $y^*_\delta = y^* + \delta y^*$. 
We will use the estimate \eqref{eq:deltaAcondition} to obtain an upper bound for the perturbation of the semigroup, e.g.\ $\lVert S_t^\delta - S_t \rVert $.
Now we introduce the perturbed problem
\begin{equation} \label{eq:problem-perturbed}
 \min_{u \in \cH} \left\{ J_\delta(u) \colon \lVert S_T^\delta u + \int_0^T S_{\tau}^\delta f_\delta (T-\tau)\drm \tau - y^*_\delta\rVert \leq \epsilon \right\}
\end{equation}
where
\[
 J_\delta (u) =  \frac{\alpha_\delta}{2} \lVert u \rVert^2 + \frac{1}{2} \int_0^T \beta_\delta (t) \left\lVert S_t^\delta u + \int_0^t S_{\tau}^\delta f_\delta(t-\tau)\drm \tau - w_\delta(t) \right\rVert^2 \drm t .
\]
Let $ \yhoms_{\delta} = y_{\delta}^* - \int_0^T S^{\delta}_{\tau}f_{\delta}(T-\tau)\drm \tau$, $\whom_{\delta} = w_{\delta} - \int_0^{\boldsymbol{\cdot}} S^{\delta}_{\tau}f_{\delta}(\boldsymbol{\cdot} - \tau)\drm \tau$, $\delta \yhoms = \yhoms_{\delta} - \yhoms$ and $\delta \whom = \whom_{\delta} - \whom$.
We denote the unique solution of the perturbed problem \eqref{eq:problem-perturbed} by $\uhat_\delta$, and recall that $\uhat$ is the unique solution of the unperturbed problem~\eqref{eq:problem}.
\begin{theorem} \label{thm:sensitivity}
 Under the above assumptions we have 
 \[
  \lVert \uhat_\delta - \uhat \rVert < C \nu
 \]
 for $\nu$ small enough, where $C$ is a constant that does not depend on $\nu$.
\end{theorem}
\begin{remark}
 Let us discuss certain situations where assumption \eqref{eq:deltaAcondition} is satisfied. 
 \begin{enumerate}[(i)]
  \item Let $\lVert \delta A \rVert < \nu$, $\zeta = \max \{\max \sigma (A) , \max \sigma (A + \delta A)\} + 1$ and $R = 1$. 

  Then, by the second resolvent identity and $\lVert (z-T)^{-1} \rVert = 1 / \operatorname{dist} (z,\sigma (T))$ for self-adjoint $T$ and $z \in \rho (T)$ we conclude for all $s \in \RR$ 
 \begin{align*}
  \lVert (\zeta + \mathrm{i} s - A_\delta)^{-1} - (\zeta + \mathrm{i} s - A)^{-1} \rVert
  &=
  \lVert (\zeta + \mathrm{i} s - A_\delta)^{-1} \delta A (\zeta + \mathrm{i} s - A)^{-1} \rVert \\
  &< \frac{\nu}{\operatorname{dist}(\zeta + \mathrm{i} s , \sigma (A_\delta)) \cdot \operatorname{dist}(\zeta + \mathrm{i} s , \sigma (A))}  \\
   & \leq \frac{\nu}{1+s^2} .
 \end{align*}
  Hence, \eqref{eq:deltaAcondition} is satisfied if $\delta A$ is a bounded operator with norm smaller than $\nu$.
  \item %
  Let $\delta A$ be relatively bounded with respect to $A$ with a relative bound smaller than $\nu$, i.e.\ $ \lVert \delta A x \rVert  \le a \lVert x \rVert + b \lVert Ax \rVert  $ for all $x\in \dom(A)$ with $0\leq b < \nu$ and a nonnegative constant $a$. Let $R=1$.
  Then  $A + \delta A$ is an upper-bounded self-adjoint operator, see e.g.\ \cite[Theorem V.4.3, Theorem V.4.11]{Kato1995}. 
  Let $\zeta > \max \sigma (A) =: \kappa$ be arbitrary.  Then we have $\lVert A (\zeta - A)^{-1} \rVert = \int_{-\infty}^{\kappa}   \frac{\lvert \lambda \rvert }{\zeta - \lambda} \,\mathrm{d}\lVert E(\lambda) \rVert  $,
  hence for all $\delta > 0$ and all $\zeta > \frac{2 + \delta}{1 + \delta}\kappa$ we have $\lVert A (\zeta - A)^{-1} \rVert < 1 + \delta$. 
  Let  $\delta = (\nu -b)/(a+b) $ and let  $\zeta > \frac{2 + \delta}{1 + \delta}\kappa$ be such that  
  $\lVert (\zeta - A)^{-1} \rVert < \delta $.  Let $x\in\dom (A)$ be arbitrary and let $y = (\zeta -A)x$.
  Then 
  \[
  \lVert \delta A x \rVert \le a \lVert (\zeta -A)^{-1}y \rVert + b \lVert A (\zeta - A)^{-1}y \rVert < \nu \lVert (\zeta - A)x \rVert .    
  \]
  Let $\chi\ge \zeta$ be arbitrary. Then from 
  $\langle x + (\chi - \zeta) (\zeta -A)^{-1}x, x \rangle \ge 1$ for all $x\in \cH$ such that $\lVert x \rVert = 1 $ it follows $\lVert (I + (\chi - \zeta)(\zeta -A)^{-1})^{-1} \rVert \le 1$ and hence 
  \[
  \lVert \delta A (\chi -A)^{-1}  \rVert  = \lVert \delta A (\zeta -A)^{-1} (I + (\chi - \zeta)(\zeta -A)^{-1})^{-1} \rVert < \nu.
  \]
   This implies  $\chi \in \sigma (A  +\delta A)$, e.g.\ $\max \sigma (A + \delta A) < \zeta$.
  By the second resolvent identity we have for all $s\in \mathbb{R}$
  \begin{align*}
  \lVert (\zeta + \mathrm{i} s - A_\delta)^{-1} - (\zeta + \mathrm{i} s - A)^{-1} \rVert & =  \lVert (\zeta + \mathrm{i} s - A_\delta)^{-1} \delta A (\zeta + \mathrm{i} s - A)^{-1} \rVert \\
  &\le \lVert (\zeta + \mathrm{i} s - A_\delta)^{-1} \rVert  \lVert \delta A (\zeta + \mathrm{i} s - A)^{-1} \rVert \\
  & = \operatorname{dist}(\zeta + \mathrm{i} s , \sigma (A_\delta))^{-1}\lVert \delta A (\zeta + \mathrm{i} s - A)^{-1} \rVert.
  \end{align*}
  For all $s\ne 0$ we have
  \begin{align*}
  \lVert  \delta A (\zeta + \mathrm{i} s - A)^{-1}  \rVert &= \lVert \delta A \left( (I + \mathrm{i} s (\zeta - A)^{-1}) (\zeta - A) \right)^{-1}\rVert \\
  & = \lVert \delta A (\zeta - A)^{-1} \left( I + \mathrm{i}s (\zeta - A)^{-1} \right)^{-1}  \rVert \\
  &< \frac{\nu}{s} \left\lVert \left( \frac{1}{\mathrm{i}s} I + (\zeta -A)^{-1} \right)^{-1}  \right\rVert  \\
  & = \frac{\nu}{s} \operatorname{dist}\left( -\frac{1}{\mathrm{i} s}, \left( \zeta - \sigma(A) \right)^{-1}  \right)^{-1} \le \frac{\nu}{\sqrt{s^2 +1}}   
  \end{align*}
  but note that the final estimate holds also for $s=0$.
  Hence we finally obtain for all $s\in \mathbb{R}$ and $\xi = \zeta +1$
  \[
  \lVert (\xi + \mathrm{i} s - A_\delta)^{-1} - (\xi + \mathrm{i} s - A)^{-1} \rVert <  \frac{\nu}{1 + s^2},
  \]  
  and hence the perturbation $\delta A$ satisfies the assumption (vi).
 \end{enumerate} 
\end{remark} \medskip


\begin{proof}[Proof of Theorem~\ref{thm:sensitivity}]
  We first estimate the perturbation bound for $\opb$.
  We  define $\Xi := \left\{ \zeta + i s \colon s \in \mathbb{R} \right\}$, where $\zeta$ is the value from the assumption (vi). Then $\Xi$ is in the resolvent sets of both $A$ and $A + \delta A$. Using the spectral calculus for generators of $C_0$-semigroups (see, for example \cite{EngelN1999}), we obtain
  \begin{align*}
    \opb &= \alpha I + \frac{1}{2\pi i} \int_0^T \beta (t) \int_\Xi \euler^{2t \lambda}(\lambda - A)^{-1}\drm \lambda \, \drm t,\\
    \opb + \delta \opb &= (\alpha +  \delta \alpha) I + \frac{1}{2\pi i} \int_0^T (\beta (t) + \delta \beta (t)) \int_\Xi \euler^{2t \lambda}(\lambda - A - \delta A )^{-1}\drm \lambda \, \drm t.
  \end{align*}
     Hence, using Fubini theorem and the resolvent formula, we obtain 
  \begin{multline}  \label{eq:3_terms}
  \delta \opb = \delta \alpha I + \frac{1}{2\pi i}  \int_\Xi  (\lambda - A - \delta A )^{-1}\delta A (\lambda - A)^{-1} \int_0^T \beta(t) \euler^{2t \lambda}  \drm t \, \drm  \lambda  \\ +  \frac{1}{2\pi i} \int_0^T \delta \beta (t) \int_\Xi \euler^{2 t \lambda} (\lambda -A - \delta A)^{-1} \drm \lambda \, \drm t.
  \end{multline}
  From $\lvert \int_0^T \beta(t) \euler^{2t \lambda}  \drm t\rvert \le \lVert \beta \rVert \frac{1}{2  \zeta  } (  \euler^{2 T \zeta} -1 ) $ for $\lambda \in \Xi$, the norm of the second term of $\delta \opb$ can be estimated from above by
  \begin{align*}
  &\frac{  \euler^{2 T \xi} -1 }{4  \xi  \pi}  \lVert \beta \rVert  \int_\Xi  \lVert(\lambda - A - \delta A )^{-1} - (\lambda - A)^{-1} \rVert \drm \lambda  \\ 
  &=
    \frac{  \euler^{2 T \xi} -1 }{4  \xi  \pi} \lVert \beta \rVert  \left( \int_{\xi - i R}^{\xi + iR}  + \int_{\Xi \setminus [\xi-iR, \xi + iR]}\right)   \lVert(\lambda - A - \delta A )^{-1} - (\lambda - A)^{-1} \rVert \drm \lambda 
    \\
   & < \frac{\nu  ( \euler^{2 T \xi}  - 1)}{2  \xi  \pi} \lVert \beta \rVert (R + R^{-1})  ,
  \end{align*}
  where in the last inequality we have used \eqref{eq:deltaAcondition}.
From
\begin{align*}
   & \left\lVert  \int_0^T \delta \beta (t) \int_\Xi \euler^{2 t \lambda} (\lambda -A - \delta A)^{-1} \drm \lambda \, \drm t  \right\rVert  \le 
   \int_0^T \lvert \delta \beta (t)  \rvert \left\lVert \int_\Xi \euler^{2 t \lambda} (\lambda -A - \delta A)^{-1} \drm \lambda \right\rVert \drm t \\
   &< \nu \int_0^T \left\lVert S_{2t}^{\delta} \right\rVert \drm t \le 
   \nu \int_0^T  \int_{- \infty}^\zeta \lvert \mathrm{e}^{2 t \lambda } \rvert \;  \mathrm{d}\lVert E_{A + \delta A} (\lambda) \rVert \drm t \le 
   \nu \frac{\euler^{2 T \xi}  - 1}{2 \zeta},
\end{align*}
we obtain that the third term in \eqref{eq:3_terms} has an upper bound $\nu ( \euler^{2 T \xi}  - 1) / (4 \zeta \pi)$. Hence we obtain 
\begin{equation*} 
  \lVert \delta \opb \rVert <  \nu \left( 1  + \frac{\euler^{2 T \xi}  - 1}{4 \zeta \pi} \left( 2 \lVert \beta  \rVert (R + R^{-1}) +1  \right)   \right).
  \end{equation*}
  To obtain an upper bound for $\lVert \delta\vb \rVert  $, we use the same steps as for $\delta \opb$, but pulling out the $L^2$-functions using H\"older's inequality.  
 For $t \ge 0$ we define the operator function
 \[
 D(t) = \frac{1}{2\pi i}\int_\Xi  \euler^{t \lambda} \left( (\lambda - A - \delta A)^{-1} - (\lambda - A)^{-1} \right)\drm \lambda.
 \]
 Note that $D(t) = S_t^{\delta} - S_t$, hence $D(t)$ is actually the perturbation of $S_t$. We calculate
 \[
 \lVert D(t) \rVert < \nu \pi^{-1} \mathrm{e}^{\zeta t} (R + R^{-1})  \text{ for all } t \ge 0.
 \]
  We first estimate 
  \[ \delta \whom = \delta w - \int_0^{\boldsymbol{\cdot}} D(\tau) f(\tau) \drm \lambda \, \drm \tau  - \frac{1}{2\pi i}\int_0^{\boldsymbol{\cdot}} \int_\Xi \euler^{\tau \lambda} (\lambda - A - \delta A)^{-1} \delta f(\tau) \drm \lambda \, \drm \tau .
  \]
 By using H\"older's inequality we estimate
  \begin{align*}
    \lVert \delta \whom(t) \rVert &\le \lVert \delta w(t)\rVert +   \left\lVert \int_0^t  D(\tau)  f(\tau) \drm \tau  \right\rVert \\
    &\qquad + \frac{1}{2\pi}\left\lVert \int_0^t \int_\Xi \euler^{\tau \lambda} (\lambda - A - \delta A)^{-1} \delta f(\tau) \drm \lambda \, \drm \tau  \right\rVert \\
    &< \lVert \delta w(t) \rVert +  \frac{ \nu}{\pi} (R + R^{-1}) \lVert f \rVert \sqrt{\frac{ \mathrm{e}^{2 t \zeta }-1}{2 \zeta}}   + 
    \frac{\nu}{2 \pi}  \sqrt{\frac{ \mathrm{e}^{2 t \zeta}-1}{2 \zeta}} . 
  \end{align*}
  This implies
  \[
    \lVert \delta  \whom \rVert <  \sqrt{2}\nu \left( 1 + \frac{1}{8 \zeta \pi^2}\left( 2 (R + R^{-1} + 1) \right)^2 \lVert f \rVert  
     \left( \frac{\mathrm{e}^{2T \zeta}-1}{2 \zeta}  - T  \right) \right)^{1/2}  . 
  \]
  Now we are in position to estimate $\delta\vb$. Since 
  \begin{multline*}
     \delta\vb = \int_0^T \beta(t) D(T+t) \whom(t)\drm t + \int_0^T \beta(t) D(T+t) \delta \whom(t)\drm t 
     \\
     + \int_0^T \delta\beta(t) D(T+t) \whom(t)\drm t  + \int_0^T \delta\beta(t) D(T+t) \delta \whom(t)\drm t 
     \\
     + \frac{1}{2 \pi i} \int_0^T \beta(t) \int_\Xi \euler^{\tau \lambda} (\lambda - A - \delta A)^{-1}  \delta \whom(t)\drm \lambda \,\drm t \\
     + \frac{1}{2 \pi i} \int_0^T \delta\beta(t) \int_\Xi \euler^{\tau \lambda} (\lambda - A - \delta A)^{-1}   \whom(t)\drm \lambda \,\drm t \\
     + \frac{1}{2 \pi i} \int_0^T \delta\beta(t) \int_\Xi \euler^{\tau \lambda} (\lambda - A - \delta A)^{-1}  \delta \whom(t)\drm \lambda \,\drm t
   \end{multline*} 
   we can again estimate $\lVert \delta\vb \rVert$ using the techniques from above and obtain
   \begin{align*}
     \lVert \delta\vb \rVert &\le  C \nu,
   \end{align*}
   where $C$ is a constant which does not depend on $ \nu$ and which may change from line to line.
   Similarly we obtain 
   \[
   \lVert \delta \yhoms \rVert \le C \nu.  
   \]
   Hence we obtained that for $\nu < 1$, each of  $\lVert \delta \opb \rVert $, $\lVert \delta\vb \rVert $  and  $\lVert \delta \yhoms \rVert $ has an upper bound of the form $C \nu$. We have also proved $\lVert D(t) \rVert  < C(t) \nu$.

   As the solution is given in terms of linear systems \eqref{eq:linsym-mu} and \eqref{eq:linsym-final}, to prove the claim of the theorem it is sufficient to show that the solutions of these systems are stable under perturbations. First note that for a chosen $\mu$ the operator on the left hand side of \eqref{eq:linsym-mu} and \eqref{eq:linsym-final} is bounded and strictly positive and that the same holds for the perturbed right hand side. Moreover, from the estimates obtained above, we see that the perturbation of the left hand side of \eqref{eq:linsym-mu} is given by 
   \[ \mu D(2T) + \delta \opb \]
    and the perturbation of the right hand side of \eqref{eq:linsym-mu} is given by 
    \[\mu S_{2T} \delta \yhoms + \mu D(2T)(\yhoms + \delta \yhoms) + S_T \delta\vb + D(T)(\vb + \delta\vb). \]
    Hence the norms of the perturbations of both left and right hand side of \eqref{eq:linsym-mu} are smaller than $C \nu$ if $\nu$ is small enough. This allows us to apply the standard perturbation theoretic results for the solutions of linear systems (\cite{Nashed1976}, see also \cite[Proposition 4.2]{ChenXue1997}) and conclude that the perturbed system \eqref{eq:linsym-mu} has the solution $x + \delta x$ with $\delta x$ satisfying $\lVert \delta  x\rVert < C \nu $. 
    Let\ $\Phi_{\delta}$ be the perturbed function $\Phi$, $\delta \Phi_{\delta} = \Phi_{\delta} - \Phi$, and let $\mu_{\epsilon}^{\delta}$ be the solution of the equation $\Phi_{\delta}(\mu) = \epsilon$. Then 
    $\Phi_{\delta} (\mu_{\epsilon}^{\delta}) = \lVert \yhoms + \delta\yhoms - x_{\epsilon} - \delta x_{\epsilon} \rVert $, where $x_{\epsilon} + \delta x_{\epsilon}$ is the solution of the perturbed system  \eqref{eq:linsym-mu} with $\mu = \mu_{\epsilon}^{\delta}$. Using the obtained bounds on the perturbations, it follows $ \lvert \delta\Phi_{\delta} (\mu_{\epsilon}^{\delta}) \rvert < C \nu $. Hence 
  \[ \lvert \Phi (\mu_{\epsilon}^{\delta}) - \Phi (\mu_{\epsilon}) \rvert = 
  \lvert \Phi_{\delta} (\mu_{\epsilon}^{\delta}) - \delta \Phi (\mu_{\epsilon}^{\delta}) - \epsilon \rvert = \lvert \delta \Phi (\mu_{\epsilon}^{\delta}) \rvert < C \nu. \]
  Since $\Phi$ is a continuous and monotone function, it follows that $\Phi^{-1}$ is continuous, hence we obtain $\lvert \delta \muhat\rvert < C \nu$. 
\end{proof}
\section{Approximations of semigroups and related operator functions}\label{Sec:rat}
In this section we will review rational approximation methods for a semigroup $S_t$ whose generator $A$ is a self-adjoint operator with upper bound $\kappa \le 0$. Our approach to constructing numerical approximations can be applied to non stable systems (those systems for which $\kappa > 0$) as well, but such systems are not included among our examples, and  we just note that in the case of a non-stable systems the estimates include a multiplicative constant which grows exponentially with $\kappa$.

We say that the function $r$ is a type $(n,m)$ rational function, where $n$ and $m$ are nonnegative integers, if there are polynomials $p$ and $q$ of degrees at most $n$ and $m$, respectively, such that $r=p/q$. Here the degrees $n$ and $m$ need not be optimal. Given a rational function $r$ we have
\begin{equation}\label{sesq}
\aligned
(v,S_t v)-(v,r(A)v)&=\int^{\kappa}_{-\infty}\big(\mathrm{e}^{t\lambda}-
r(\lambda)\big)\,\mathrm{d}(E(\lambda)v,v),
\endaligned    
\end{equation}
where $E(\cdot)$ denotes the spectral measure of the self-adjoint operator $A$ \cite{ReedSimonII,Kato1995}. The support of the spectral measure of any of the operators $tA$, for $t>0$ is contained in $(-\infty, 0]$ and computing similarly as in \eqref{sesq} we obtain the estimate
\begin{equation}\label{SC}
\lVert g(A)v-r(A)v \rVert \leq \lVert g-r \rVert_{L^\infty(-\infty, 0]} \lVert v \rVert ,
\end{equation}
where the function $g$ is measurable with respect to the spectral measure of $A$. When considering numerical efficiency, a key information is if there exists a rational approximation of $g$ with $n$ small.

For the case in which $g(z)=e^{-z}$, it is known, see \cite{CODY196950,Trefethen2006}, that  for each $n$ there exists a unique type $(n,n)$ rational function $r^*_n$ 
which minimizes  $\lVert g-\cdot \rVert_{L^\infty(-\infty,0]}$. Furthermore, there exists a constant $C>0$, independent of $n$, such that $r^*_n$ verifies
\begin{equation}\label{Halphen}
\aligned
\lVert g-&r^*_n \rVert_{L^\infty(-\infty,0]}\\&=\min\{ \lVert g-r \rVert_{L^\infty(-\infty,0]}~:~r \text{ is a type} (n,n) \text { rational function.} \}\\&\leq\frac{C}{H^n}\leq\frac{C}{9.28903^n} .
\endaligned\end{equation}
The number $H$ is known under the name of Halphen constant, see \cite{STAHL2009821}. The rational function $r^*_n$ is the unique minimizer of $\lVert g-r \lVert_{L^\infty(-\infty,0]}$ among $(n,n)$ rational functions. Further, $r^*_n$ does not have zero-pole pairs appearing on the negative real axis. 

For the error analysis of approximations of semi-groups it is particularly convenient if the rational function is representable in the partial fractions form. For constants $r_0$ and $r_i$, $\zeta_i$, $i=1,\cdots,d$ the expression
\[
\widehat{r}(z)=r_0+\frac{r_1}{z-\zeta_{1}}+\cdots\frac{r_d}{z-\zeta_d}
\]
is a partial fractions expansion of the rational function $\widehat{r}$. It has been shown that for $g(z)=e^{-z}$ one can construct, see \cite{Trefethen2006}, a partial fractions expansion of the type $(n,n)$ rational function $\widehat{r}_n$ such that $\lVert g-\widehat{r}_n \rVert_{L^\infty(-\infty,0]}\leq C 3.2^{-n}$. The constant $C>0$ is independent of $n$.
The poles $\zeta_i$ are contained on a hyperbola in a complex plane and the weights are defined by the application of the $n$ point quadrature rule to the Cauchy integral representation of the exponential function with this hyperbola as a contour.

\subsection{Rational function fitting}
\label{sect:rff}
To approximate solutions of the constrained parabolic control problem, we will need rational approximations of slightly more general functions. Let us first note that the optimal approximation result \eqref{Halphen} can be extended, see  \cite{STAHL2009821}, in a slightly modified form to the class of perturbed exponential functions $g$ which can be represented as
\begin{equation*}
g(x)=u_0(x)+u_1(x)\mathrm{e}^{ax}
\end{equation*}
where $u_0$ and $u_1\neq 0$ are arbitrary rational functions and $a<0$. According to \cite[Theorem 1]{STAHL2009821}, for any $n\in\NN$ and a chosen but fixed integer $k$ such that $n-k\geq 0$ there exists a unique rational function $r_{n,n+k}^*$ such that 
\[
r_{n,n+k}^*=\text{arg min}\{\lVert g-r_{n,n+k} \rVert_{L^\infty (-\infty , 0]}~:~ r~\text{ is rational function of type } (n,n+k)\}
\]
and $\lVert g-r^*_{n,n+k} \rVert_{L^\infty(-\infty,0]}\leq C 9.28903^{-n}$.  Further, the results of \cite{CODY196950,STAHL2009821} are existential. A way to construct a rational approximation satisfying  \eqref{Halphen} is to transform the interval $(-1,1]$ to $(-\infty,0]$ and then apply the contour integration technique to the transformed problem, see \cite{Trefethen2006}. This can be achieved by the Moebius transformation $m(z)=9(z-1)/(z+1)$, see \cite{Trefethen2006}. The inverse transformation to $m$ is given by the formula $m^{-1}(z)=-(z+9)/(z-9)$ and it maps $\soi{-\infty}{0}$ to $\soi{-1}{1}$. Then the function to approximate is $g(z)=\mathrm{e}^{m(z)}:(-1,1]\to\mathbb{R}$ and the rational function which approximates $e^z$ is obtained by composing the rational approximant of $g$ with the inverse Moebius transformation. We first loop a finite contour around the interval $(-1,1]$ and then these points get mapped by the Moebius transform into points on a curve looping around the infinite interval $(-\infty,0]$.

Let now $g_1$ and $g_2$ be perturbed exponential functions. We are interested in finding type $(n,n)$ rational approximations of functions of the form $g_1+g_2$, $g_1 g_2$, $g_1/g_2$ and $g_i\circ m$. Obviously, combining rational approximations $r_i$ of $g_i$ is a natural first idea. However, the rational functions $r_\diamond=r_1+r_2$, $r_\diamond=r_1 r_2$ or $r_\diamond=r_1/r_2$ will in general be of a different (component-wise larger) type.

We can however use an approximation approach to truncate the type of the product, sum or a quotient of two rational functions of the type $(n,n)$ to a rational function $\widetilde{r}_\circ$ of the type $(n,n)$ which for given $\texttt{tol}>0$ and an interval $\ci{a}{b}$ satisfies the estimate $\lVert\widetilde{r}_\diamond-r_\diamond  \rVert_{L^2\ci{a}{b}}\leq \texttt{tol} \lVert r_\diamond \rVert_{L^2\ci{a}{b}}$. 

To this end we use the award winning \texttt{rkfit} algorithm from \cite{BK17}. This is the rational Krylov function fitting algorithm which implements the rational functions calculus by working with a representation of a rational function as a transfer function of a pencil of Hessenberg matrices. It performs all ot the aforementioned operations (addition, division, multiplication and composition with a Moebius transformation) stably using only floating point arithmetic. According to \cite{BK17} given a tolerance $\mathtt{tol}$ and the perturbed exponential function $g(x)=u_0(x)+u_1(x)\mathrm{e}^{ax}$, $a<0$ such that $g\in L^2(-\infty,0]$, \texttt{rkfit} algorithm produces a rational function \begin{equation}\label{pf}
r_{RK}(x)=r_0+\frac{r_1}{x-\tilde{\zeta}_1}+\cdots+\frac{r_d}{x-\tilde{\zeta}_d} ,
\end{equation}
in the pole residue form, such that 
\begin{equation}\label{rkf}
\lVert r_{RK}-g \rVert_{L^\infty(-\infty,0]}\leq\mathtt{tol} \lVert g \rVert_{L^2(-\infty,0]} .
\end{equation}
We can now construct the operator $r_{RK}(A):=r_0 I+\sum_{i=1}^dr_i(A-\tilde{\zeta}_i)^{-1}$
such that
\[
\lVert g(A)-r_{RK}(A) \rVert_{L(\cH)}\leq\texttt{tol} \lVert g \rVert_{L^2(-\infty,0]} .
\]
\subsection{Galerkin resolvent estimates}
The steps needed to compute the action of a function of an operator on a vector, exemplary $g(A)v=S_tv$, involve two steps. First, we approximate the function $g$ by a rational function on an interval containing the spectrum of the self-adjoint operator $A$. We then need to sample the resolvent $(z-A)^{-1}v$ at the poles of the rational function $r$. 

In what follows we will restrict our considerations to the operator of the divergence type posed in a compact polygonal domain $\Omega\subset\R^2$. Many statements are algebraic in nature and hold in a more general setting. However, the interpolation results for piecewise polynomial functions and the regularity results for the domain of the operator are specific to the aforementioned class of operators.
 
We approximate the action of the resolvent by selecting a finite dimensional subspace $\mathcal{V}_h\subset\dom(A^{1/2})$ and then forming the Galerkin projection of $A$ onto $\mathcal{V}_h$. According to \cite[Section 5]{Lasiecka1}, the Galerkin projection $A_h:\mathcal{V}_h\to\mathcal{V}_h$ is given by the formula
\[
A_h=(A^{1/2}P_h)^*(A^{1/2}P_h),
\]
where $P_h$ is the orthogonal projection onto $\mathcal{V}_h$. Let $\mathcal{V}_h$ be the space of piece-wise linear, for a given triangular tessellation of $\Omega$, and continuous functions on $\Omega$. The resolvent estimate for $A$ using the Galerkin projection $A_h$ reads (see e.g. \cite{MR4011540} for technical details)
\begin{equation}\label{res_est}
\lVert (z-A)^{-1}v-(z-A_h)^{-1}v \rVert_{L^2(\Omega)}\leq Ch^{2\nu} \lVert v \rVert_{L^2(\Omega)} ,
\end{equation}
for $h<h_0$ and $v\in V_h$. Here $\nu>0$ is a parameter depending on the regularity of the functions in $\dom(A)$ and $h$ is the maximal diameter of a triangle in the chosen tessellation of $\Omega$ and $h_0$ is denoting the minimal level of refinement from which the estimate holds. Note that constants $C$ and $h_0$ do depend on $z$ in an explicit way but do not depend on $v$, see \cite{MR4011540}. We will, however need this estimate solely for at most $d$ poles $\tilde{\zeta}_i$, $i=1,\cdots,d$ of the rational function $r_{RK}$ from \eqref{pf}, and so
\[
\lVert r_{RK}(A)v-r_{RK}(A_h)v \rVert_{L^2(\Omega)}\leq d~C~h^{2\nu} \lVert v \rVert_{L^2(\Omega)} .
\]
Finally, let $g(x)=u_0(x)+u_1(x)\mathrm{e}^{ax}$ be the perturbed exponential function. Based on \eqref{SC}, for a given rational function $r_{RK}$ and $v\in V_h$ we have the estimate
\begin{align*} 
\lVert g(A)v &- r_{RK}(A_h)v \rVert_{L^2(\Omega)}
\\
&\leq \lVert g(A)v-r_{RK}(A)v \rVert_{L^2(\Omega)}+ \lVert r_{RK}(A)v-r_{RK}(A_h)v \rVert_{L^2(\Omega)}\\
&\leq \lVert g-r_{RK} \rVert_{L^{\infty}(-\infty,0]} \lVert v \rVert +dCh^{2\nu} \lVert v \rVert_{L^2(\Omega)} .
\end{align*}
By choosing suitable $r_{RK}$ and $h$, the last estimate ensures a good approximation of $g(A)v$ based on a finite dimensional approximation of  the operator $A$. 
\section{Numerical examples} 
\label{sec:examples}
In this section we consider several constrained optimization problems in 1D and 2D.
The problems are academic and are primarily chosen to test the efficiency of the developed approach. We compare our results with those obtained by other, already existing methods where such a comparison is possible. We will also report the timings as means to get an intuition of the efficiency of implementation. The timings will be reported for the workstation running Intel Core i5 8600K at 3.60 GHz with 24 GB of DDR4 ram. 

In all  examples we take the weight function  of the form $\beta=\Eins_{[T/3,2T/3]}$, while the desired trajectory $w$ is assumed to be time independent. This implies that we want the optimal state to be close to  $w$ for times $t$ between $T/3$ and $2T/3$, while no desired trajectory is prescribed outside this interval. 
With this setting and under the additional assumption that the operator $A$ is strictly negative, the operator $\opb$ and the vector $\vb$ from the main theorem can be computed explicitly as
  \begin{align*}
  \opb &= \alpha I +  \frac{1}{2} A^{-1} S_{2T/3}(I - S_{2T/3}), \quad
  \vb =  A^{-1} S_{T/3} (I - S_{T/3})w .
  \end{align*}
 We can now use spectral calculus to exemplary represent the operator $\Psi$ as
\[\aligned
\Psi &= \alpha I + \int_{\RR}\mathrm{e}^{\lambda~2T/3}~g(\lambda)\,\mathrm{d} E(\lambda)\\
&= \alpha I +\int_{\RR} \mathrm{e}^{\lambda~2T/3} (1/\lambda - \mathrm{e}^{\lambda 2T/3}/\lambda) \,\mathrm{d} E(\lambda) .
\endaligned
\]
The function $\lambda \to 1/\lambda -\mathrm{e}^{\lambda~2T/3}/\lambda$ is obviously the perturbed exponential function for which the rational approximation theory holds (there exists a small degree rational approximation). We can equivalently use rational approximation theory to compute the vector $\vb$. In numerical procedure the first  step is to determine $\mu_\eps$ - the solution to the equation $\Phi(\mu)=\eps$ (cf. \eqref{eq:Phi}). Taking into account the properties  of $\Phi$ (given in Lemma \ref{lem:Phi}), the equation has a unique solution for every $\eps\in (0,\Phi(0))$. Any root finding algorithm based only on function evaluation  can be used to robustly approximate the root $\varepsilon_0$. We use the Brent method as it is implemented in the Matlab's procedure \texttt{fzero}. We keep the convergence criterion for the root finding procedure below the discretization error for the finite element approximation. According to the resolvent analysis, the error in the approximation by a rational function is a lower order perturbation of the system, as compared to the discretization error.

The value of the function $\Phi(\mu)$ is computed by using the rational approximation and the spectral calculus
\[\aligned
(\mu S_{2T} &+  \opb)^{-1} \mu S_{2T} \yhoms \\
&=\int_{-\infty}^0\frac{\mu \mathrm{e}^{2T\lambda}}{\mu \mathrm{e}^{2T\lambda} +\alpha + (1/\lambda -\mathrm{e}^{\lambda~2T/3}/\lambda)} \,\mathrm{d} E(\lambda)\yhoms\\
&\approx r_0\yhoms +\sum_{i=1}^dr_i(\tilde{\zeta}_i-A)^{-1}\yhoms .
\endaligned\]
The function 
\[
g(\lambda)=\frac{\mu \mathrm{e}^{2T\lambda}}{\mu \mathrm{e}^{2T\lambda} +\alpha + (1/\lambda -\mathrm{e}^{\lambda~2T/3}/\lambda)}
\]
is approximated on $(-\infty,0]$ using the rational function $r$ with $18$ pole residue pairs. The approximation $r$ satisfies \eqref{rkf} with $\mathtt{tol}=10^{-15}$, as it is the default for \texttt{rkfit}. The function is a quotient of perturbed exponential functions for which we know that there is a high quality low degree rational approximation. We could compute a rational approximation of $g$ as a quotient of rational approximations, however this rational function could have, in the worst case, double the degree of the best rational approximation of the numerator and denominator. Instead, as discussed in Section \ref{sect:rff}, we choose to approximate the function $g$ directly, as means of keeping the degree of the approximating rational function lower. Note that these approximations are obviously independent of $A$, and hold on $(-\infty,0]$.  

Once, the equation for  $\mu_\eps$ is solved, the optimal initial control $\uhat$ follows by  \eqref{sol_form}. For its computation we explore the same procedure as the one for calculation of $\Phi(\mu)$.

\subsection{1D heat equation}
As the first test of the proposed method we consider the heat equation (with variable coefficient) on $\Omega=[0,\pi]$ accompanied by homogeneous Dirichlet boundary conditions.
The operator $A$  is taken of the form
\[
A=-\partial_x((1+a \Eins_{\left[\gamma,\pi\right]})\partial_x)
\]
with $\gamma=2.2$. The parameter $\gamma$ determines the contact of two materials with a different diffusivity coefficient. We consider two cases: 
\begin{enumerate}
\item $a=0$, with $A$ being the isotropic Laplace operator;
\item $a=-0.8$, resulting in discontinuity of diffusion coefficient at point $\gamma$. 
\end{enumerate}
Operator $A$ is discretized by conforming linear finite elements with $h=1/20$ and we use the lumped mass discretization in order to be able to utilized optimized \texttt{rkfit} library.

Besides the function $\beta$ determined in the beginning of this section, for this example we propose 
\begin{itemize}
\item $\alpha = 10^{-4}$,
\item final time $T=0.01$, 
\item desired trajectory $\omega = \Eins_{[\pi/5,2 \pi/5]}$,
\item final target $y^* = \Eins_{[3 \pi/5, 4 \pi/5]}$,
\item $f=0$ (homogenous equation). 
\end{itemize}
The choice of $w$ stimulates the state trajectory to be concentrated on the left part of the domain during the central time period, while  at the final time the target $y^*$ requires it to be supported at the right hand side, at least for small values of the tolerance $\eps$. 

For the isotropic case $a=0$, the above setting coincides with Example 4.1 from \cite{LazarMP-17}. In such a way we shall be able to compare our results with those obtained by a different method based on spectral decomposition of the Laplace operator. 

The example is performed for three values of the final tolerance 
\[\eps=[0.2, 0.5, 0.9] \Phi(0),\]
depicted on Figure \ref{Phi}, together with the corresponding values $\mu_\eps$ (solutions to equation \eqref{eq:Phi}) and graph of function $\Phi$ (in $\log-\log$ scale). 
The figure confirms the properties of function $\Phi$ provided by  Lemma \ref{lem:Phi}. Specially, its initial value $\Phi(0)$ coincides with 
\[\lVert \ytilde - y^* \rVert =1.0374\]
where $\ytilde$ is the optimal final state of the unconstrained problem.  The corresponding initial value $\utilde$,  which is just the minimizer of functional $J$ can also be obtained by standard methods of convex analysis. 

\begin{remark}
Note that a gradient method for computing the minimizer of $J$ requires a solution of the forward problem for the parabolic equation. A basic step of any implicit method for the solution of a parabolic equation is the evaluation of a resolvent like function. The convergence of a gradient method with Nesterov's acceleration is at best $1/n^2$, where $n$ is the number of forward problem solves. On the other hand, our method is based on utilizing functional calculus and the best rational function approximations of operator functions. In consequence, these operator functions can be approximated with a method that converges at the rate of at least $9^{-n}$, where $n$ is the number of the resolvent evaluations. This is a very crude comparison. However, under a very modest assumption that we need at least one evaluation of the resolvent function per forward problem solve, we clearly see a potential advantage of the the rational function approach. This is the reason why such methods are becoming methods of choice for the solution of parabolic problems and also for numerically inverting the Laplace transform in the case of the solution of inverse problems, see \cite{STAHL2009821}. 
\end{remark} \medskip
The elapsed time to produce the plot which included sampling $\Phi$ in $350$ points was $12.97$ seconds and it took $0.36$ seconds to compute $\Psi(0)$ alone.

\begin{figure}[ht]
	\centering
	 \subfloat[{Unisotropic difusion}]{\includegraphics[width = 5cm]{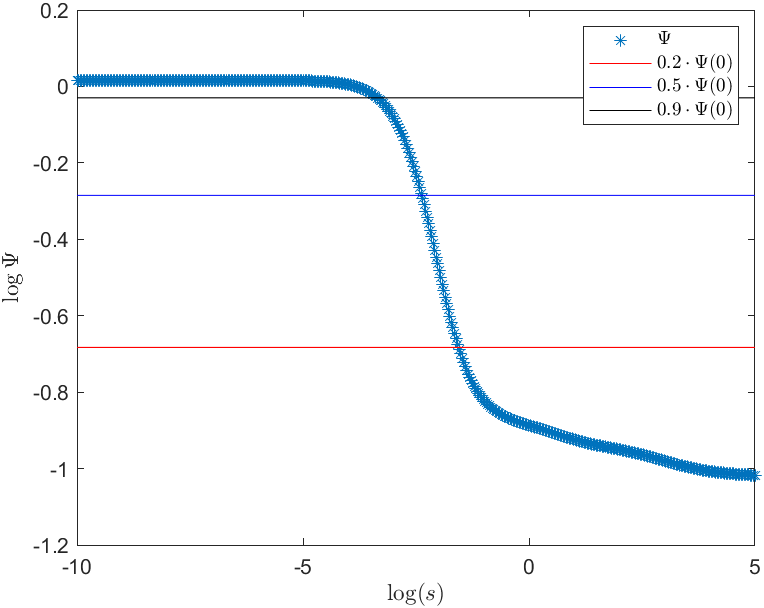}}
	\subfloat[{Isotropic difusion}]{\includegraphics[width = 5cm]{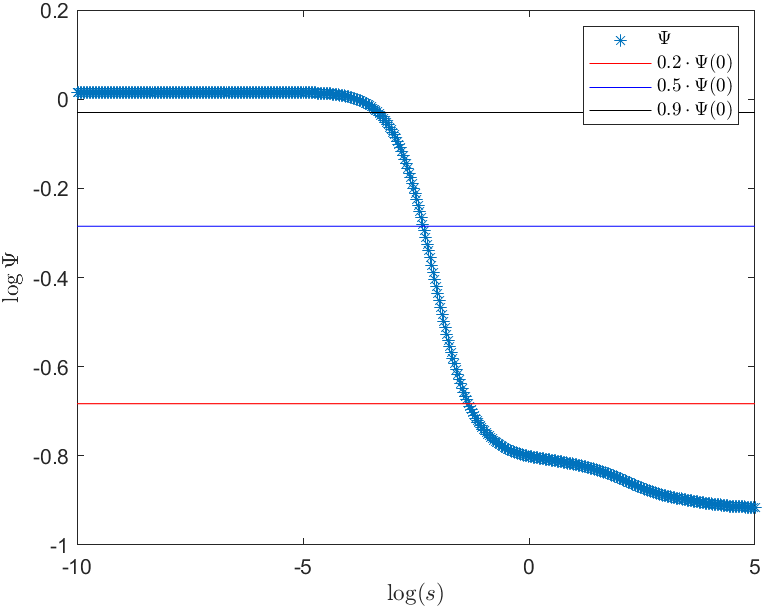}}
	\caption{Function $\Phi$, the chosen values of $\eps$ and corresponding $\mu_\eps$. }
\label{Phi}
\end{figure}
The results in the isotropic case for the prescribed values of $\eps$ are presented in Figure \ref{Fig1D}. 
\begin{figure}[ht]
	\centering
	\includegraphics[width = 4cm]{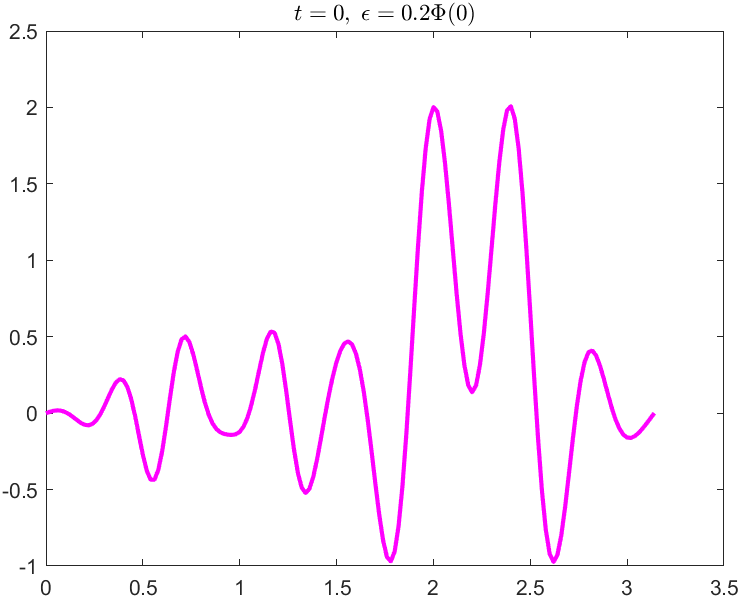}
	\includegraphics[width = 4cm]{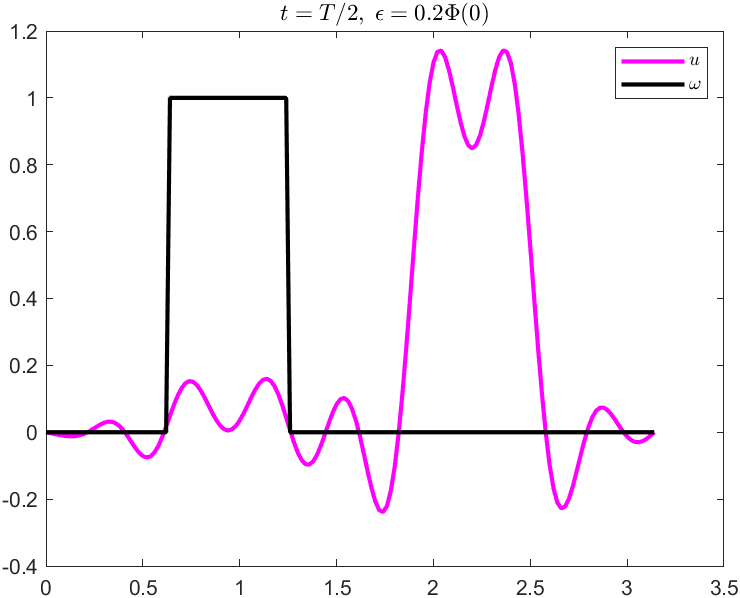}
	\includegraphics[width = 4cm]{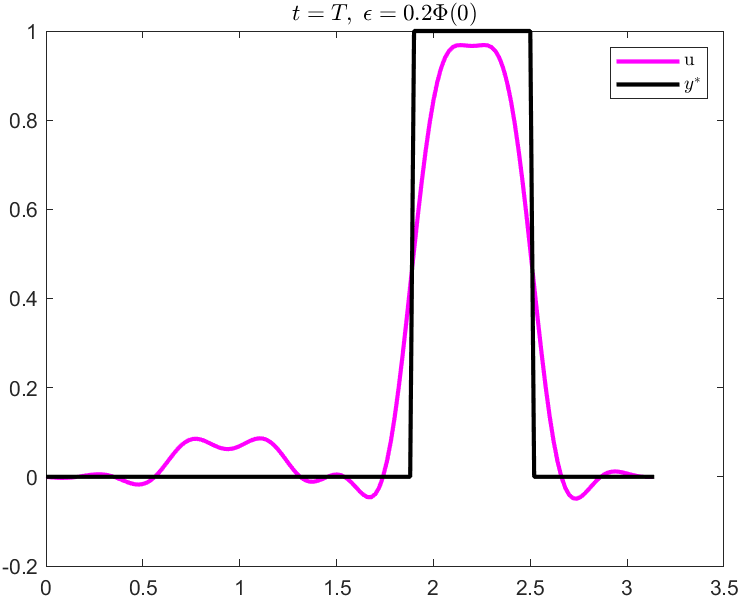}\\
	\includegraphics[width = 4cm]{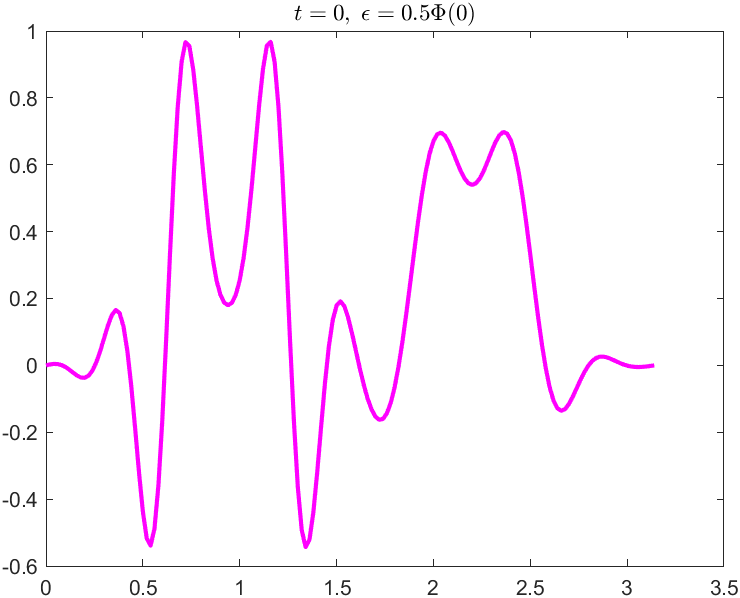}
	\includegraphics[width = 4cm]{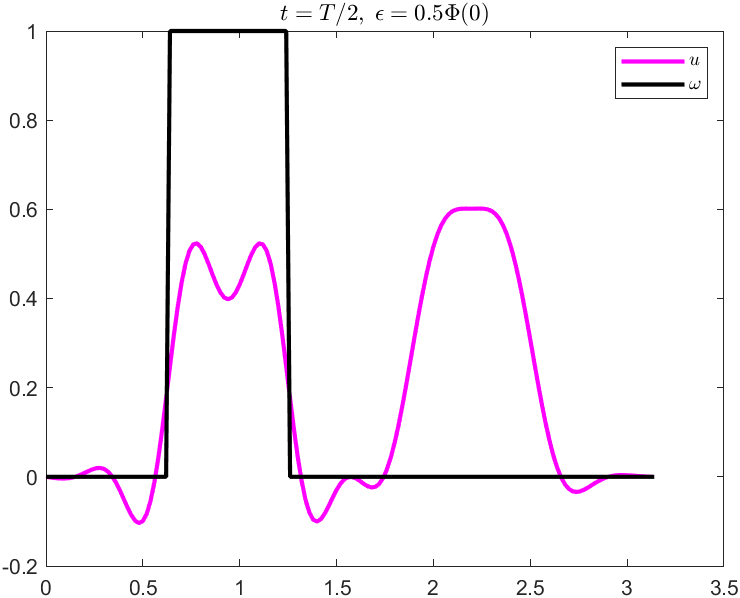}
	\includegraphics[width = 4cm]{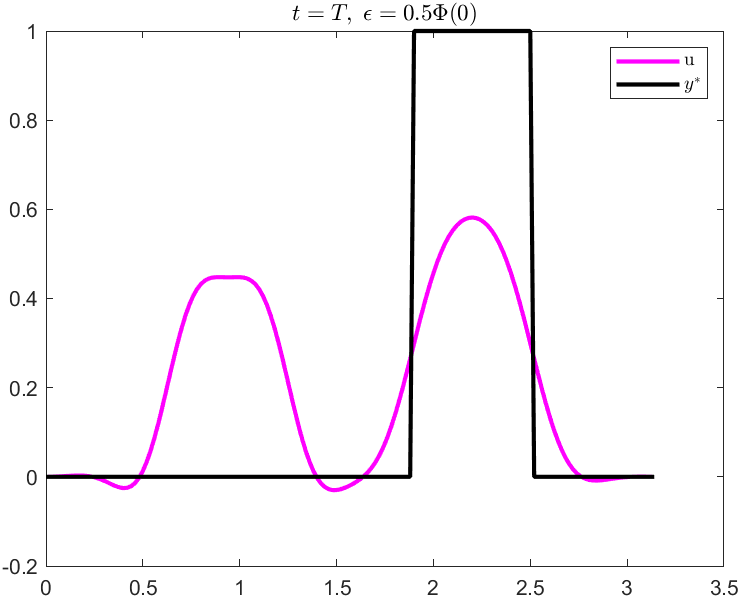}\\
	\includegraphics[width = 4cm]{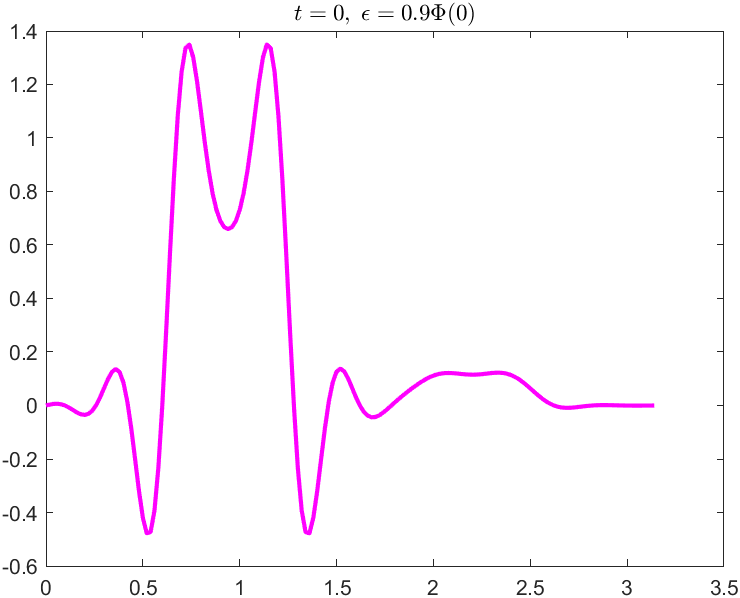}
	\includegraphics[width = 4cm]{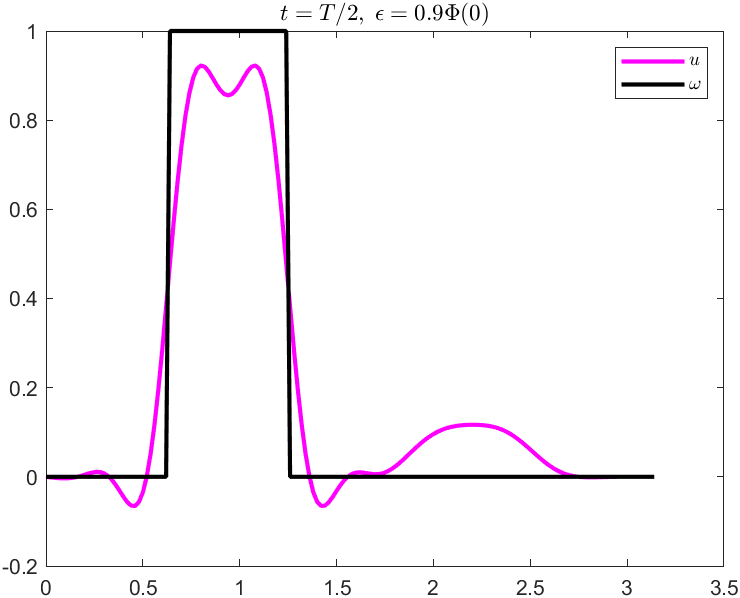}
	\includegraphics[width = 4cm]{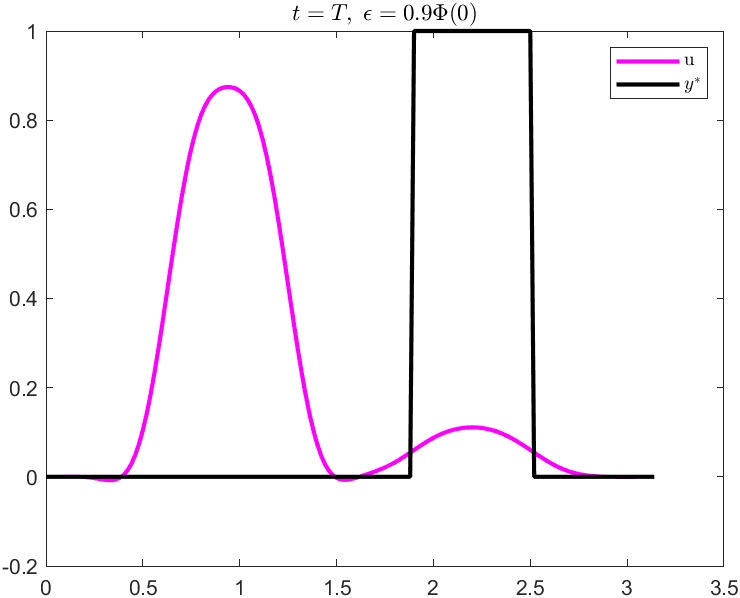}\\
		\caption{Example 5.1, isotropic case. The initial control $u=y(0)$ (left), the computed solution at time $t=T/2$ compared with the desired trajectory $\omega$ (middle), and the optimal final state at $t=T$ compared with the target $y^*$ (right) for three different values of the tolerance $\epsilon$.}
	\label{Fig1D} 
\end{figure}
When $\eps$ is small  the initial mass is concentrated on the support of the  target $y^*$,  in order to steer the system close to it at the final time. 
On the opposite, for large value of $\eps$, the initial control is concentrated on the left. In such a way the solution stays close to the desired trajectory $\omega$ in the middle part of the time interval 
during which the distributed cost $\beta$ is active. Finally, the intermediate value of $\eps$ is a trade-off between the optimisation of the cost functional $J$ and the requirement to hit the final target with the given tolerance. 

Besides agreeing with the intuition, the results completely coincide with those obtained in \cite[Example 4.1]{LazarMP-17}. This provides the first confirmation of the  method proposed in this article. Furthermore, we note that for increased  tolerance $\eps$ the optimal control resembles  the solution of the unconstrained problem $\utilde$ (Figure \ref{sol-unconstr}), where the latter is just a minimizer of functional $J$. This is expected, as for large values of $\eps$ the solution is less affected by the prescribed target $y^*$, while the unconstrained problem is completely independent of it.
\begin{figure}[ht]
	\centering
\includegraphics[width = 6cm]{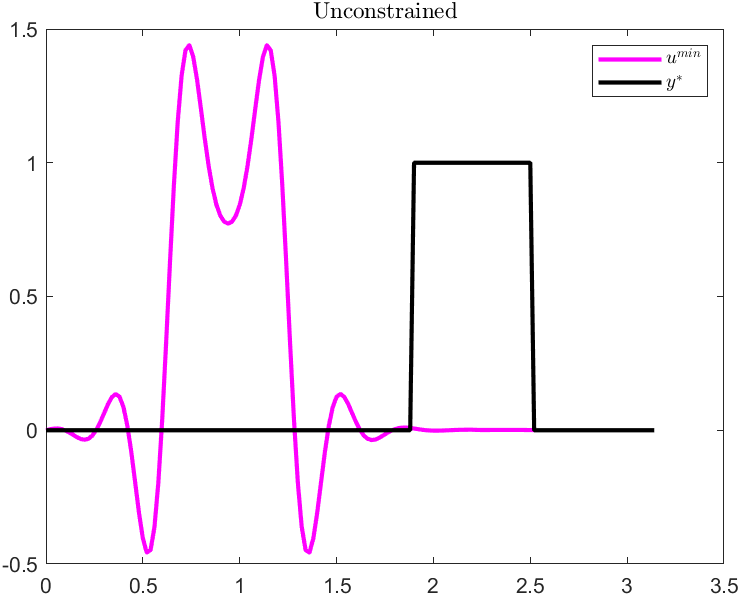}
	\caption{Example 5.1, unisotropic case. The plot of $\utilde$ and $y^*$.}
\label{sol-unconstr} 
\end{figure}
The results for the discontinuous diffusion and for the same range of the final tolerance are presented in Figure \ref{Fig1D-disc}. 
\begin{figure}[ht]
	\centering
	\includegraphics[width = 4cm]{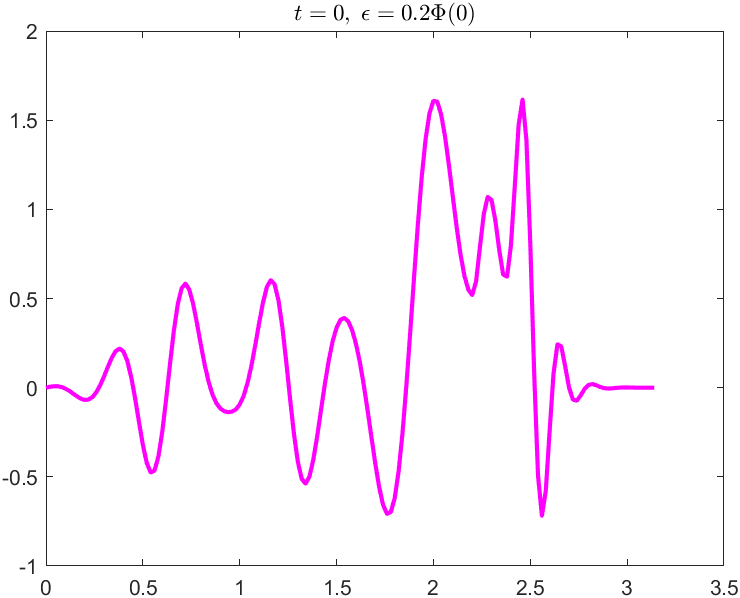}
	\includegraphics[width = 4cm]{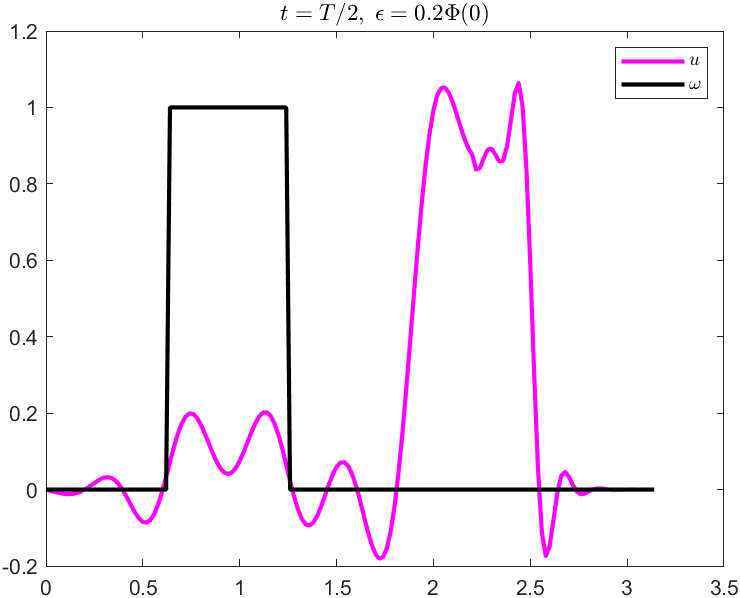}
	\includegraphics[width = 4cm]{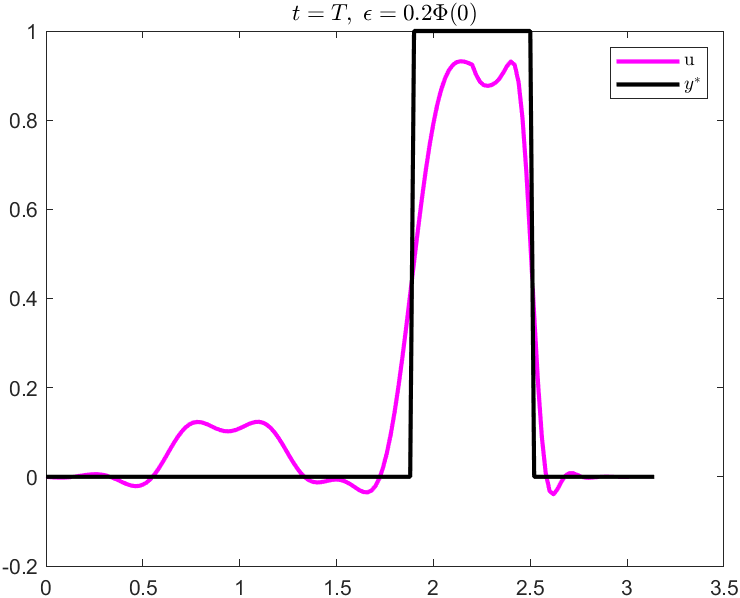}\\
	\includegraphics[width = 4cm]{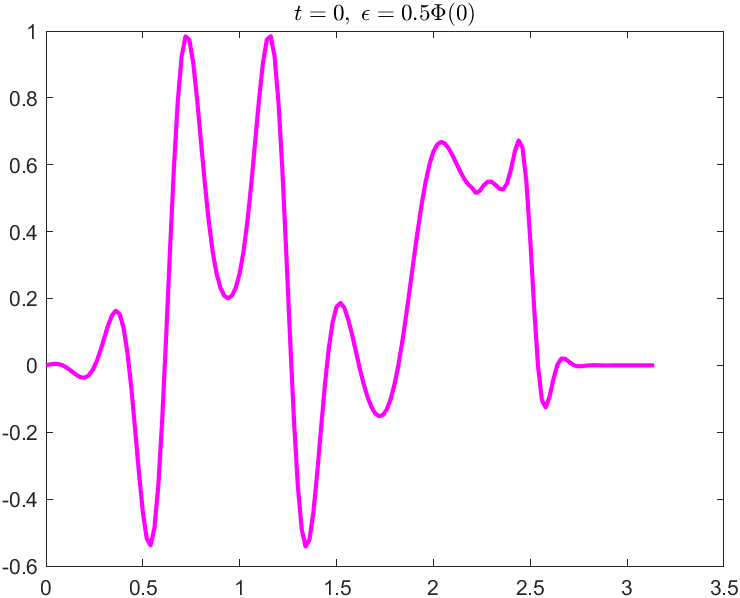}
	\includegraphics[width = 4cm]{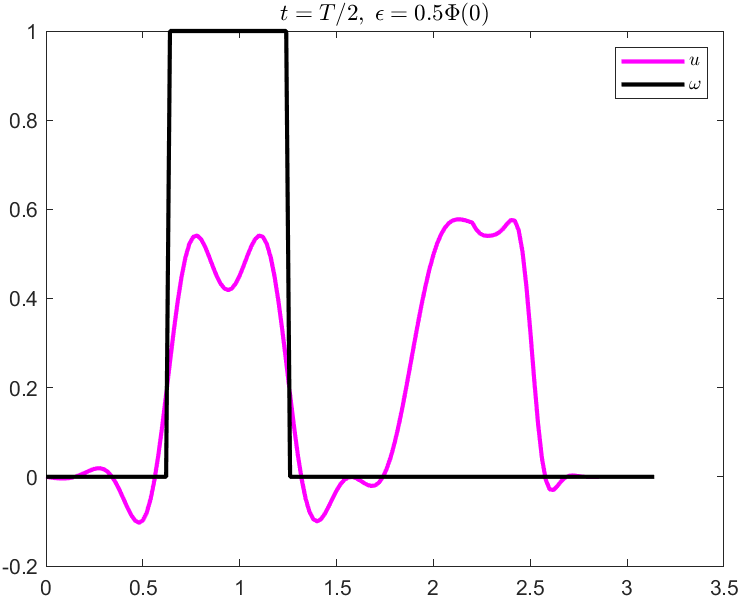}
	\includegraphics[width = 4cm]{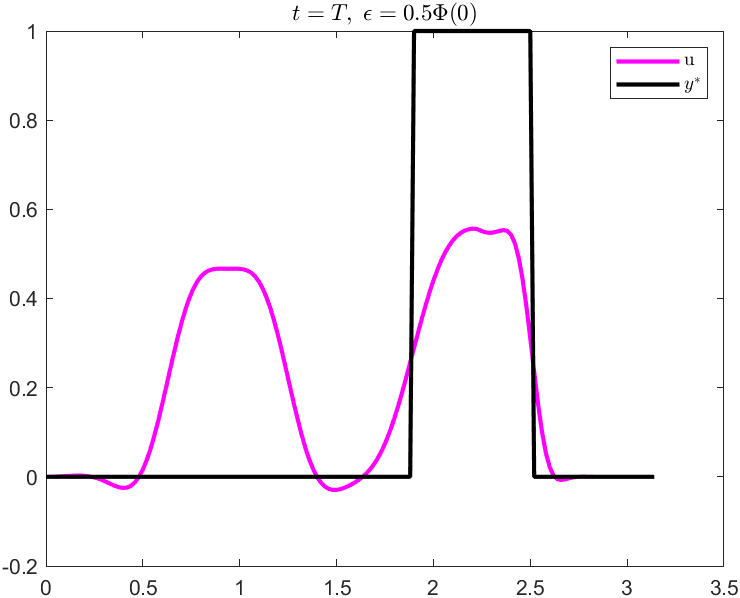}\\
	\includegraphics[width = 4cm]{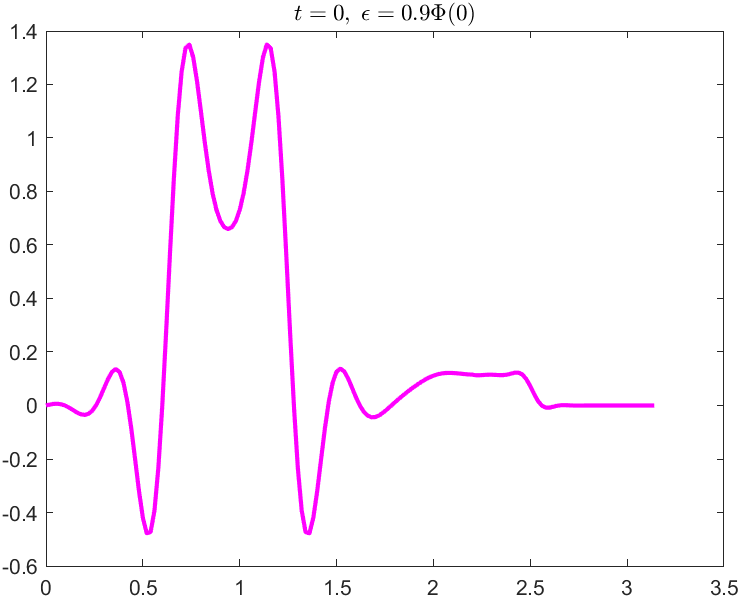}
	\includegraphics[width = 4cm]{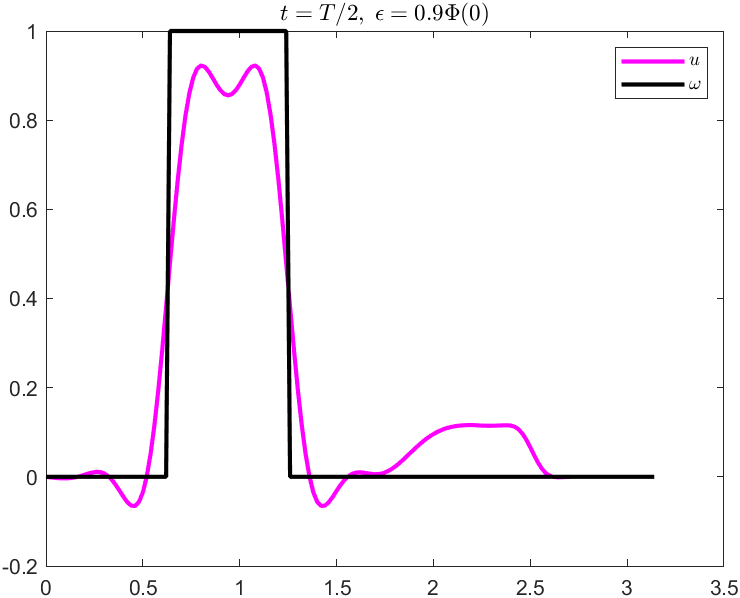}
	\includegraphics[width = 4cm]{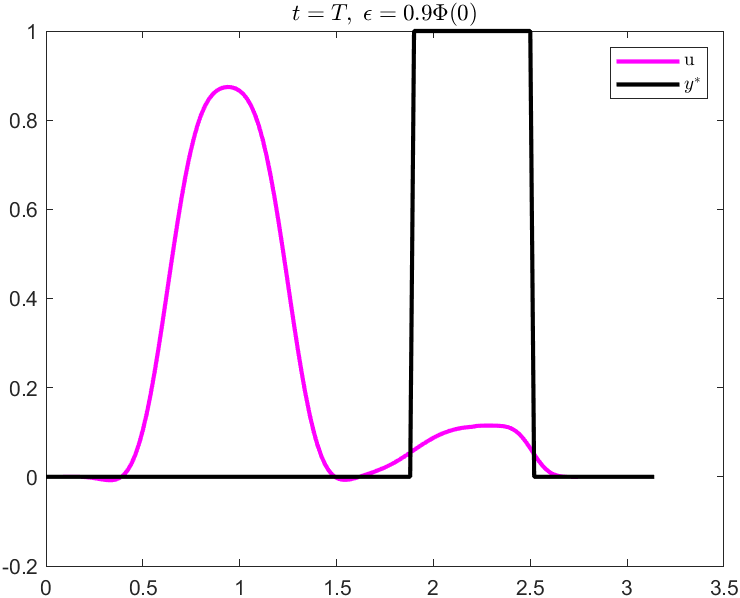}\\
	\caption{Example 5.1, discontinuous diffusion. The initial control $u=y(0)$ (left), the computed solution at time $t=T/2$ compared with the desired trajectory $\omega$ (middle), and the optimal final state compared with the target $y^*$ (right) for three different values of the tolerance $\epsilon$.}
	\label{Fig1D-disc} 
\end{figure}
In their main features, the results coincide with those obtained in the case of the constant diffusion coefficient. The novelty is broken symmetry of the solution in the right part of the domain, where the discontinuity occurs. As a consequence, the center of initial mass is slightly shifted  rightward, where  diffusion processes are slower. This is logical, having in mind that in this region the initial mass can better approximate the characteristic function (of the support of $y^*$) during a larger period of time, as small diffusion rate will not modify its form significantly. 


\subsection{2D heat equation on irregular domain}
In the next example we repeat the same calculation in the 2D setting. To this end we use the Dirichlet Laplace operator defined on the L-shape domain $\Omega=\left[-1,1\right]^2\setminus(\left[-1,0\right]\times\left[0,1\right])$. We will use Lagrange P1 elements (i.e. we approximate using piecewise linear and continuous functions). We have used a shape regular mesh with $h=1/30$, also with the lumped mass discretization of the semigroup, and we choose $T=1/20$. Due to the reentrant corner of the L-shaped domain we have a loss of regularity of the functions in $\dom(A)$ and so the resolvent estimate \eqref{res_est} holds with $\nu$, $0<\nu<1$. In the case of $H^2$-regular solutions we would have $r=1$. 

For the target data we choose 
\begin{itemize}
\item $\omega(x)=\Eins_{\lVert x-x_0 \rVert_1\leq 0.2}$,
\item $y^*(x)=\mathrm{e}^{-20 \lVert x-x_1 \rVert^2} + \mathrm{e}^{-20 \lVert x-x_2 \rVert^2} + \mathrm{e}^{-30 \lVert x-x_3 \rVert^2}$,
\end{itemize}
with $x_0=(-0.5,-0.5)$, $x_1=(0.5,0.5)$, $x_2=(0.6,0.1)$ and $x_3=(0.8,0.4)$ (Figure \ref{2D-tar-fig}). The other parameters are the same as in the previous example.

\begin{figure}[ht]
	\centering
	\includegraphics[width = 6cm]{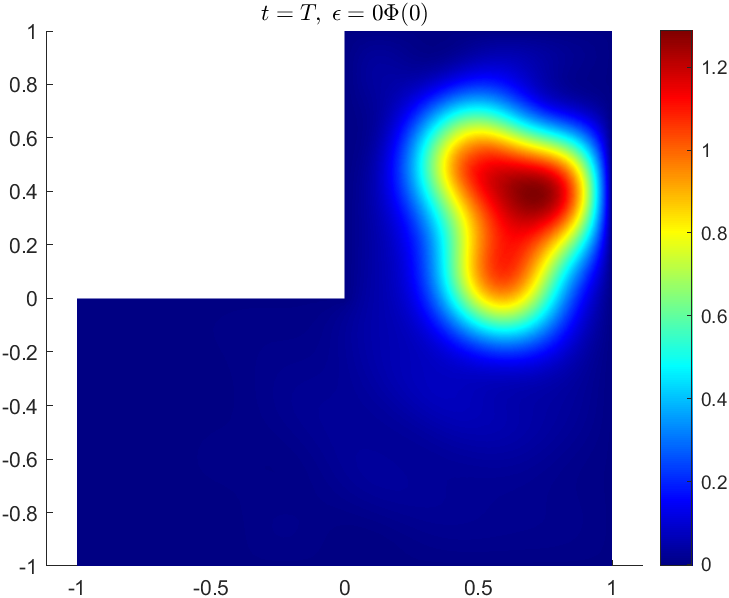}
		\caption{Example 5.2. The prescribed final target $y^*$.}
	\label{2D-tar-fig}
\end{figure}


The results for the three values of the final tolerance $\eps=[0.1, 0.5, 0.9] \Phi(0)$ are displayed in Figure \ref{2D-fig}. We show  the solutions' snapshots at $t=0, T/2, T$.
\begin{figure}[ht]
	\centering
	\includegraphics[width = 4.2cm]{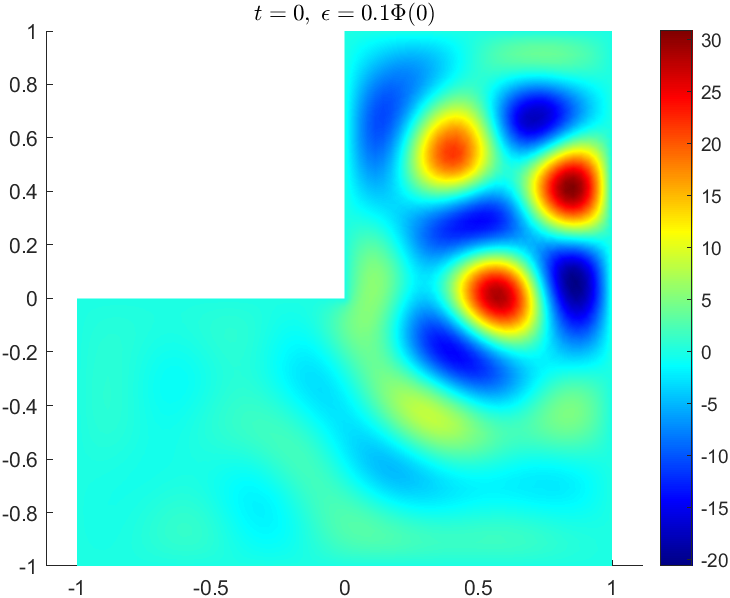}
	\begin{tikzpicture}
	\node[inner sep=0pt] (initial) at (0,0)
	{\includegraphics[width = 4.2cm]{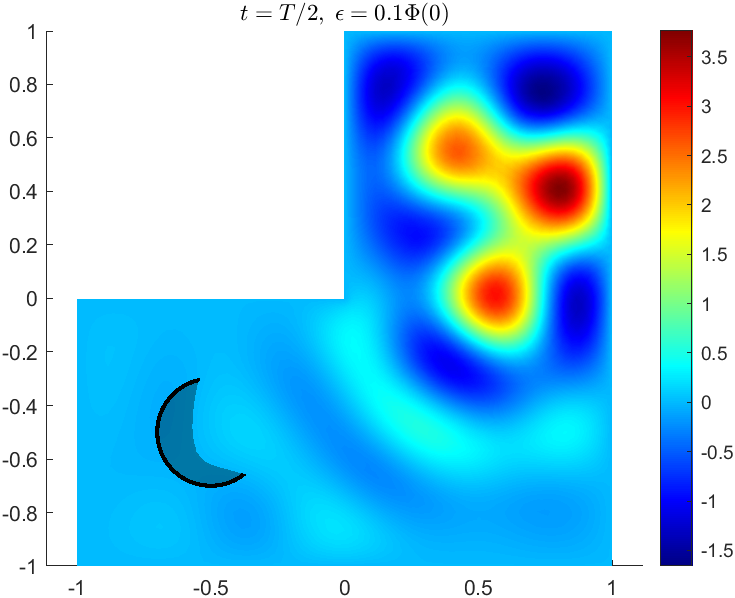}};
	\draw[red,thick,dashed] (-0.9,-0.77) circle (0.31cm);
	\end{tikzpicture}
	\includegraphics[width = 4.2cm]{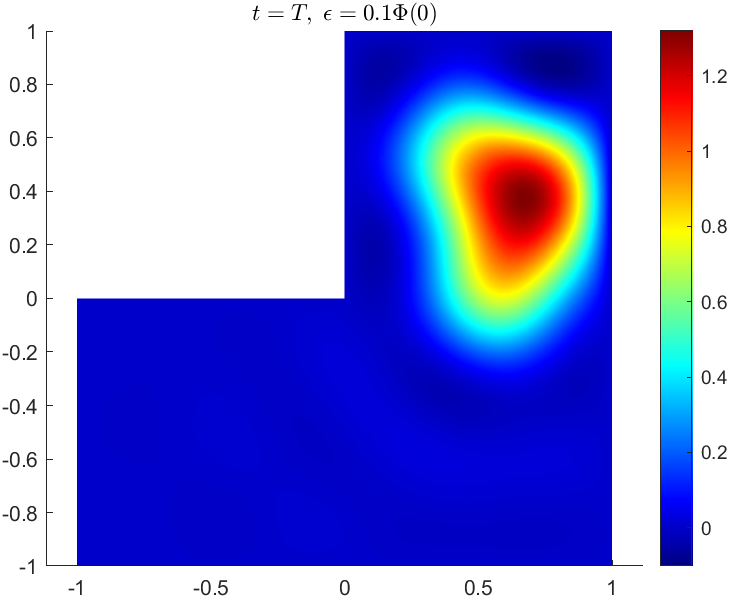}\\
	\includegraphics[width = 4.2cm]{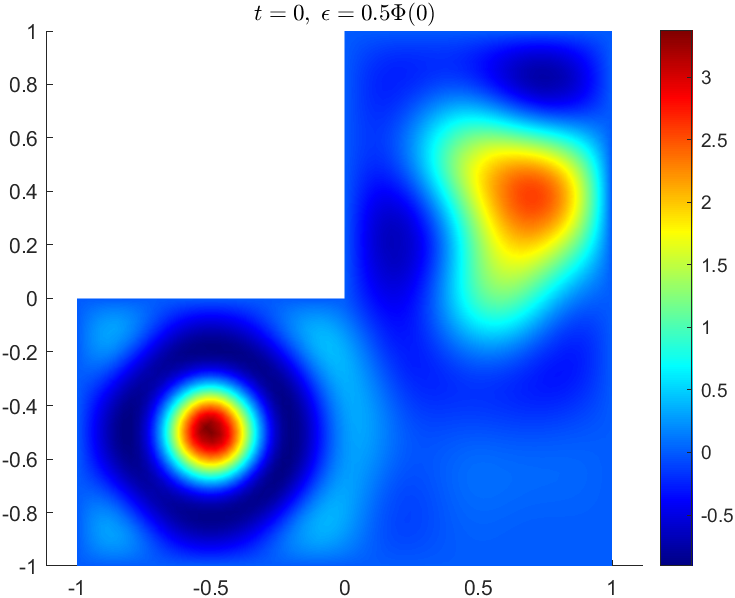}
	\begin{tikzpicture}
	\node[inner sep=0pt] (initial) at (0,0)
	{\includegraphics[width = 4.2cm]{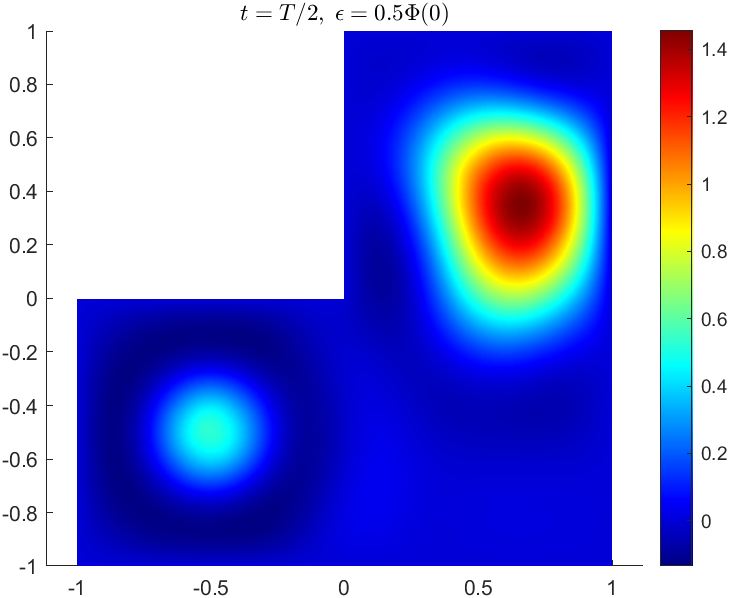}};
	\draw[red,thick,dashed] (-0.9,-0.77) circle (0.31cm);
	\end{tikzpicture}
	\includegraphics[width = 4.2cm]{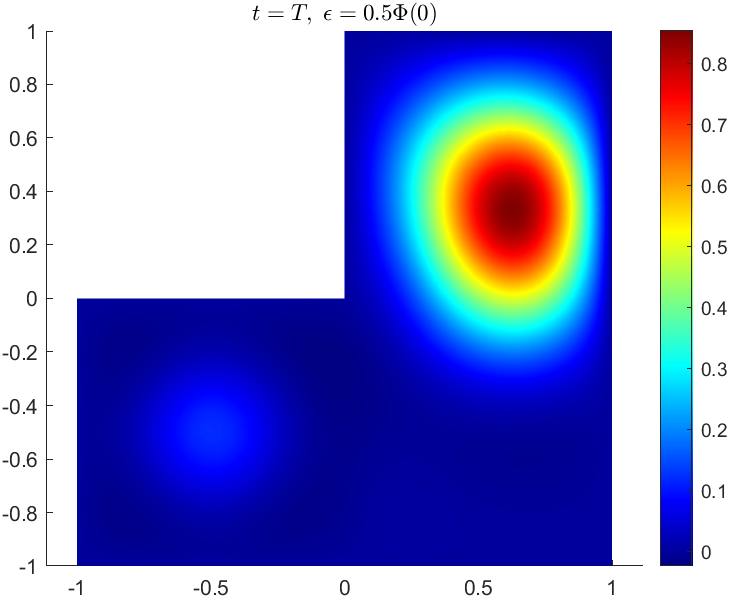}\\
	\includegraphics[width = 4.2cm]{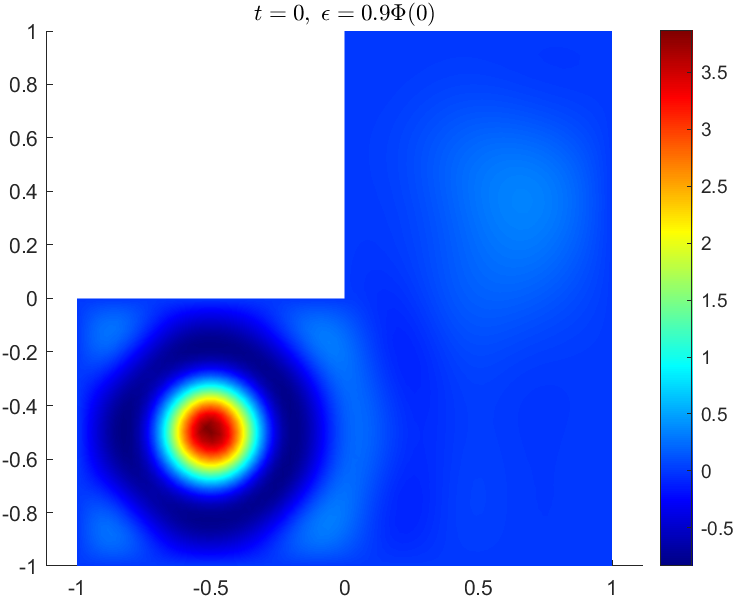}
	\begin{tikzpicture}
	\node[inner sep=0pt] (initial) at (0,0)
	{\includegraphics[width = 4.2cm]{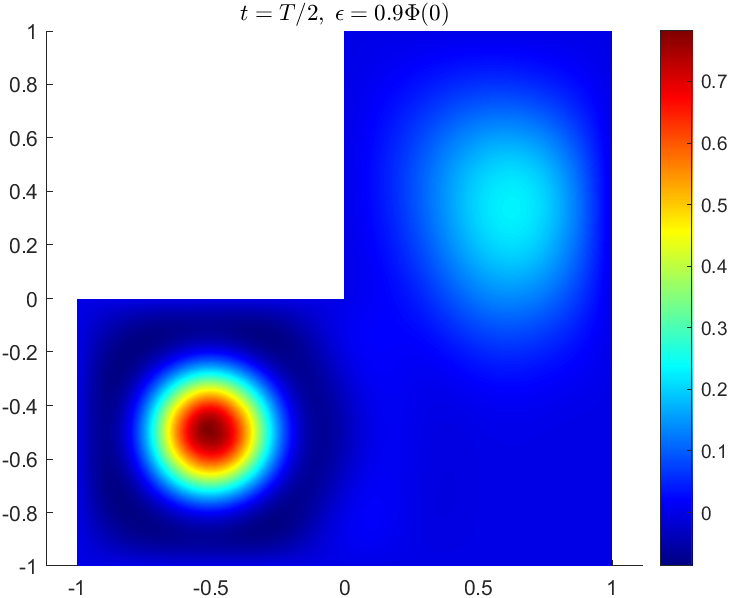}};
	\draw[red,thick,dashed] (-0.9,-0.77) circle (0.31cm);
	\end{tikzpicture}
	\includegraphics[width = 4.2cm]{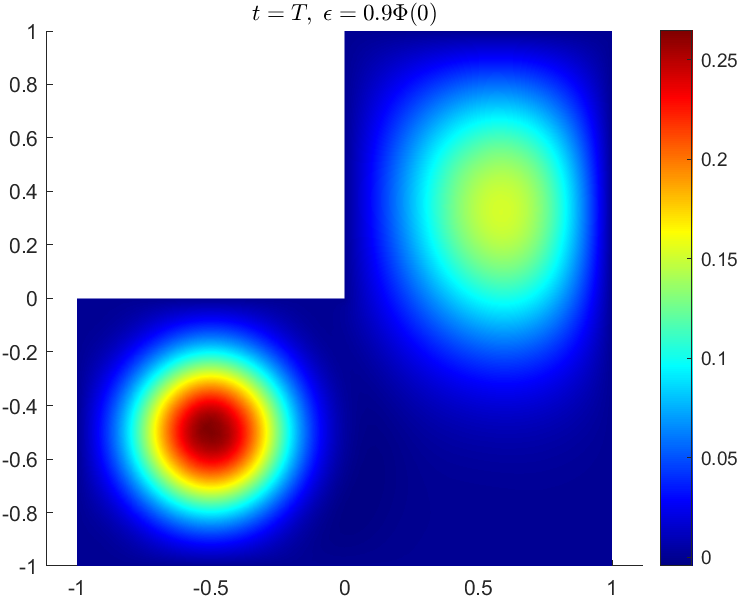}\\
	\caption{Example 5.2. The initial control $u=y(0)$ (left), the computed solution at time $t=T/2$ compared with the desired trajectory $\omega$ (middle), and the optimal final state (right) at $t=T$ for three different values of the tolerance $\epsilon$. The red dashed circle marks the constraint $\omega$ on the trajectory.}
	\label{2D-fig}
\end{figure}

The first row depicts evolution of the state for small tolerance $\eps$. The initial control steers the system close to the prescribed target $y^*$ (cf. Figure \ref{2D-tar-fig}) at the final time, while there is no coincidence with $\omega$ in the between period. For the tolerance $0.5$ of the range of $\Phi$ equal importance is assigned both to $\omega$ and the final state. The largest tolerance allows the solution to optimize the given cost functional almost independently of the prescribed target $y^*$. 

Essentially, the results  exhibit the same behavior as those obtained in the previous examples.  The elapsed time for computing $\Psi(0)$ -- the unconstrained problem -- is $0.9828$ seconds. This  demonstrates efficiency and flexibility of the method in 2D and, in particular, in case of irregular domains. We chose to exemplary report the timing for $\Psi(0)$, since in this case it is possible to compute the value of $\Psi$ with other methods such as those which are based on gradient optimization.

\section{Conclusion}

In this paper we have constructed and implemented a numerical algorithm for a constrained optimal control problem. The problem consists of identifying an initial datum that minimizes a given cost functional and steers the system at the final time within a prescribed distance from the target. The algorithm results in an (almost explicit) formula for the solution, expressed in terms of the operator governing the system. The formula itself was derived  previously (cf. \cite {LazarMP-17}), but its implementation was based on spectral decomposition, which requires knowledge or construction of eigenfunctions of the operator.

The main novelty of this article is twofold. Firstly, we provide a complete quantified sensitivity analysis of the solution with respect to all the data entering the problem. In particular, it implies a good approximation of the solution in cases where the operator or the external source are not completely determined. Secondly, for the numerical implementation we  explore efficient Krylov subspace techniques that allow us to approximate a complex function of an operator by a series of linear problems. We provide a-priori estimates for the approximation that are not sensitive to any particular spatial discretization, and neither to a matrix representation of the operator $A$. 
The theoretical results are confirmed by numerical examples. The first and the simplest example coincides with the one analysed in \cite {LazarMP-17}, and the results obtained with two approaches are in complete agreement. The following, more complex examples confirm the good performance of the algorithm in the case of operators with variable coefficients and acting on irregular domains.

The proposed approach can be generalised to other optimal control problems. The first step in this direction would be to consider a distributed control problem, i.e. one in which a control enters the equation through a non-homogeneous term and is active along the entire time frame. This would also allow for boundary control problems which, by using the classical Fattorini's approach \cite {F68}, can be expressed as distributed ones. The second generalisation would consider different norms that enter into the cost functional. In particular, it would be tempting to include $L^1$-terms in the cost, since these introduce sparsity into the control. Of course, such a generalisation requires a more subtle theoretical analysis, as the cost functional is not differentiable in this case. This approach would  also enable to consider different non-smooth, convex functionals.  



\section*{Acknowledgment}
The work of L.G.\ has been supported by Hrvatska Zaklada za Znanost (Croatian Science Foundation) under the grant  IP-2019-04-6268 - Randomized low rank algorithms and applications to parameter dependent problems. The work of M.L.\ and I.N.\ has been supported by Hrvatska Zaklada za Znanost (Croatian Science Foundation) under the grant  IP-2016-06-2468 - Control of Dynamical Systems.

\end{document}